\documentclass[a4paper,12pt,reqno,oneside]{amsart}
\usepackage{setspace} 
\usepackage{typearea} 

\usepackage{amscd,amsmath,amssymb,verbatim}
\usepackage{xr-hyper}
\usepackage{graphicx}
\usepackage{ifthen}
\usepackage{paralist,cite}
\usepackage[colorlinks,backref,final,bookmarksnumbered,bookmarks]{hyperref}
\usepackage[backrefs]{amsrefs}
\usepackage{nicematrix}
\usepackage{tikz-cd}
\usepackage{accents}
\usepackage{enumitem}

\usepackage[utf8]{inputenc}
\usepackage[T1,T2A]{fontenc}

\newcommand{\arxiv}[2][]{\ifthenelse{\equal{#1}{}}
{\href{http://arxiv.org/abs/#2}{\tt arXiv:#2}}
{\href{http://arxiv.org/abs/math/#2}{\tt arXiv:math.#1/#2}}}

\theoremstyle{plain}
\newtheorem{maintheorem}{Theorem}
 
\newtheorem{theorem}{Theorem}[section]
\newtheorem{lemma}[theorem]{Lemma}
\newtheorem{corollary}[theorem]{Corollary}
\newtheorem*{corollary*}{Corollary}
\newtheorem{proposition}[theorem]{Proposition}

\theoremstyle{definition}
\newtheorem{definition}[theorem]{Definition}
\newtheorem{example}[theorem]{Example}

\newtheoremstyle{remark}
{}{}{}{}{\itshape}{}{ }{\thmname{#1}\thmnumber{ \itshape #2.}}
\theoremstyle{remark}
\newtheorem{remark}[theorem]{Remark}

\def\N{\mathbb{N}} 
\def\R{\mathbb{R}} 
\def\Z{\mathbb{Z}}

\def\x{\times}
\def\but{\setminus}

\def\phi{\varphi} 

\def\xr#1{\xrightarrow{#1}} 
 \renewcommand{\:}{\colon}

\DeclareMathOperator*{\colim}{colim}

\DeclareMathOperator{\Int}{Int} \DeclareMathOperator{\id}{id}

 \DeclareMathOperator{\st}{st}
\DeclareMathOperator{\lk}{lk}

\def\ost{\mathop{\mathring{\text{st}}}}

\def\tph#1{\raise2.5pt\hbox{\the\textfont1\char"7F}\!\!#1}
\def\tpm#1{\raise0pt\hbox{\the\textfont1\char"7F}\!#1}
\def\tpl#1{\lower1.5pt\hbox{\the\textfont1\char"7F}\!#1}

\def\bydef{\mathrel{\mathop:}=}

\DeclareSymbolFont{bskadd}{U}{bskma}{m}{n}
\DeclareFontFamily{U}{bskma}{\skewchar\font130 }
\DeclareFontShape{U}{bskma}{m}{n}{<->bskma10}{}
\DeclareMathSymbol{\varlrttriangle}{\mathord}{bskadd}    {"E4}
\def\skewtriangle{\mbox{$\varlrttriangle$}}
\def\sktr{{\mathchoice  
{\scalebox{1.5}\skewtriangle}
{\scalebox{1}\skewtriangle}
{\scalebox{0.8}\skewtriangle}
{\scalebox{0.5}\skewtriangle} }}

\makeatletter
\newcommand*\nullseq{\mathop{\mathpalette\@biguoperator{\bigsqcup}}}
\newcommand*\@biguoperator[2]{\ooalign{\hidewidth$#1$\raisebox{1pt}{\scriptsize$\star$}\hidewidth\cr$#1#2$\cr}}
\makeatother
\begin{document}

\title{Fine shape III: $\Delta$-spaces and $\nabla$-spaces}
\author{Sergey A. Melikhov}
\address{Steklov Mathematical Institute of Russian Academy of Sciences,
ul.\ Gubkina 8, Moscow, 119991 Russia}
\email{melikhov@mi-ras.ru}

\begin{abstract} 
In this paper we obtain results indicating that fine shape is tractable and ``not too strong'' even in 
the non-locally compact case, and can be used to better understand infinite-dimensional metrizable spaces and
their homology theories.

We show that every Polish space $X$ is fine shape equivalent to the limit of an inverse sequence of simplicial maps
between metric simplicial complexes.
A deeper result is that if $X$ is locally finite dimensional, then the simplicial maps can be chosen 
to be non-degenerate.
They cannot be chosen to be non-degenerate if $X$ is the Taylor compactum.
\end{abstract}

\maketitle
\section{Introduction}

A {\it $\Delta$-space} $X$ is the limit of an inverse sequence $\dots\xr{p_1}|K_1|\xr{p_0}|K_0|$ of simplicial maps between 
metric simplicial complexes $K_i$.
This definition is due to Kodama \cite{Ko1}.
If the bonding maps $p_i$ are non-degenerate,%
\footnote{A simplicial map is called non-degenerate if it has no point-inverses of dimension $>0$.}
we call $X$ a {\it $\nabla$-space}.
Every $\nabla$-space is the union of its finite-dimensional skeleta, and is easily seen to be homotopy equivalent to their 
mapping telescope, which is a locally finite-dimensional $\nabla$-space (see Corollary \ref{lfd-nabla}(b)).

\begin{example}
Not every compactum is fine shape equivalent to a $\nabla$-space.

Indeed, let $X$ be the Taylor compactum (see \cite{EH}*{5.5.10}, \cite{DS}*{10.3.1}, \cite{Sak3}). 
Thus if $M$ is the mapping cone of a degree $p$ map $\phi\:S^{2q-1}\to S^{2q-1}$, where $p$ is an odd prime,
$q$ is sufficiently large and $r=2(p-1)$, then 
$X=\lim\big(\dots\xr{\Sigma^{2r}f}\Sigma^{2r}M\xr{\Sigma^r f}\Sigma^rM\xr{f}M\big)$, where $f$ is a map 
that induces an isomorphism on reduced complex $K$-theory (such a map was constructed by J. F. Adams).

Suppose that $X$ is fine shape equivalent to a $\nabla$-space $Y$.
As noted above, $Y$ is homotopy equivalent to a locally finite-dimensional metrizable space $Z$.
Hence $X$ is fine shape equivalent to $Z$.
In particular, there exist fine shape morphisms $X\xr{u}Z\xr{d}X$ whose composition is the identity morphism.
Since $X$ is compact, it follows immediately from the definition of fine shape (see \cite{M-I}*{\S\ref{fish:fineshape}}) 
that $u$ factors as $X\xr{w}K\xr{j}Z$, where $K$ is a compact subset of $Z$ and $j$ is the fine shape class
induced by the inclusion.
Since $Z$ is locally finite-dimensional and $K$ is compact, $K$ is finite-dimensional.
Let $n=\dim K$.

Since $\Sigma^{kr}M$ is $(n+1)$-connected for large $k$, the fine shape (=strong shape) morphism $K\xr{j}Z\xr{d}X$
is represented by a constant map $K\to X$.
Hence $\id_X$ represents the strong shape class of a constant map $X\to X$.
In other words, $X$ has the strong shape of a point, which is a contradiction.
\end{example}

\begin{maintheorem} \label{main3} (a) Every Polish space $X$ is fine shape equivalent to a Polish $\Delta$-space $Y$.

(b) Moreover, there is a cell-like perfect map $Y\to X$ which is a fine shape equivalence.
\end{maintheorem}

The compact case of (a) is due to Kodama \cite{Ko1} (taking into account that for compacta fine shape coincides 
with shape on the level of isomorphism classes of objects, see \cite{M1}*{Proposition 2.6}).
He also obtained the finite-dimensional compact case of (b) \cite{Ko3} (building on 
the construction of Kaul \cite{Kaul}) and similar results for finite dimensional locally compact metrizable spaces 
(regarding their Borsuk shape and their Fox shape) \cite{Ko2}, \cite{Ko3}.

\begin{maintheorem} \label{main4} Every locally finite-dimensional Polish space $X$ is fine shape equivalent 
to a locally finite-dimensional $\nabla$-space.
\end{maintheorem}

In the case where $X$ is finite-dimensional, a bit more can be said (see Theorem \ref{nabla}(a)).
This finite-dimensional version of Theorem \ref{main4} will be applied to the study of axiomatic homology in 
the upcoming revision of the preprint \cite{M-II}.

The proof of Theorem \ref{main4} involves combinatorial constructions which may be of independent interest.
Given a simplicial complex $K$, we introduce its ``non-degenerate resolution'', which is a simplicial 
complex $\hat K^\infty$ endowed with a simplicial map onto the barycentric subdivision $K^\flat$ of $K$ such that 
every simplicial map $K\to L$ between simplicial complexes (or rather the induced map $K^\flat\to L^\flat$) 
lifts to a {\it non-degenerate} simplicial map $\hat K^\infty\to\hat L^\infty$.
We prove that if $K$ is finite-dimensional, then $\hat K^\infty$ deformation retracts onto a copy of $K^\flat$
(Theorem \ref{def-retr2}); a simplified version of this result is used to obtain Theorem \ref{main4}.
The deformation retraction does not seem to be easily obtainable by usual techniques of algebraic topology
and results from explicit simplicial collapses constructed by a multi-layered induction.
At the start of this construction we encounter a remarkable subdivision of a simplex into products of
simplexes (see Example \ref{grayson}).

The beginning of the paper is devoted to a study of the topology of $\Delta$-spaces (\S\ref{topology}) and of their
special kinds (\S\ref{special}), which does not involve any references to fine shape.

\section{Topology of $\Delta$-spaces} \label{topology}

\subsection{Simplexes and Hilbert simplexes}

We recall that a {\it $\Delta$-space} $X$ is the limit of an inverse sequence $\dots\xr{p_1}|K_1|\xr{p_0}|K_0|$ of simplicial maps 
between metric simplicial complexes $K_i$.%
\footnote{Here we speak of ``metric simplicial complexes'' rather than ``polyhedra'' so as to ensure that they come endowed 
with the more rigid additional structure. (See \cite{M00} for the definitions.)}

\begin{lemma} \label{surjective}
(a) The bonding maps $p_i$ can be chosen to be surjective.

(b) $X$ is the union of all the inverse limits of the form $\dots\xr{f_1}\sigma^{n_1}\xr{f_0}\sigma^{n_0}$, 
where each $\sigma^{n_i}\in K_i$, each $f_i$ is the restriction of $p_i$ and a surjective map.
\end{lemma}

\begin{proof}[Proof. (a)] Let $L_i$ be the smallest subcomplex of $K_i$ such that $|L_i|$ contains the image of $X$.
Then clearly $p_i(L_{i+1})=L_i$.
\end{proof}

\begin{proof}[(b)] Given an $x=(x_0,x_1,\dots)\in X$, let $\sigma^{n_i}$ be the smallest simplex of $K_i$
containing $x_i$.
Then clearly $p_i(\sigma^{n_{i+1}})=\sigma^{n_i}$.
\end{proof}

We will regard a $\Delta$-space $X$ as a space endowed with an additional structure, namely, the 
representation of $X$ as the union of all the inverse limits as in (b).
All these inverse limits are either simplexes (when $\lim n_i$ exists) or {\it Hilbert simplexes} (otherwise).
It is easy to see that the vertex set of a Hilbert simplex (that is, the inverse limit of the vertex sets 
of the simplexes $\sigma^{n_i}$) can be an arbitrary $0$-dimensional compactum of infinite cardinality.

\begin{lemma} \label{infty-simplex}
Each Hilbert simplex of a $\Delta$-space contains a copy of the Hilbert cube.
\end{lemma}

\begin{proof}
The given Hilbert simplex $\sigma^\infty$ is an inverse limit of simplicial surjections 
$\dots\to\sigma^{n_1}\to\sigma^{n_0}$ between $n_i$-simplexes such that $n_i\to\infty$ as $i\to\infty$.
By omitting some members of the inverse sequence we may assume that each $n_{i+1}>n_i$.
It is easy to see that every surjection between finite sets factors into a composition of
surjections between finite sets whose cardinalities differ by $1$.
Hence $\sigma^\infty$ is a copy an inverse limit of simplicial surjections of the form
$\dots\xr{p_2}\Delta^2\xr{p_1}\Delta^1\xr{p_0}\Delta^0$.
It is easy to see that for every $i$ there is an embedding $g\:\Delta^i\x I\to\Delta^{i+1}$ such that 
the composition $\Delta^i\x I\xr{g}\Delta^{i+1}\xr{p_i}\Delta^i$ is the projection.
Hence there are embeddings $f_i\:I^i\to\sigma^i$ such that every diagram
\[\begin{CD}
I^{i+1}@>f_{i+1}>>\sigma^{i+1}\\
@V\pi_iVV@Vp_iVV\\
I^i@>f_i>>\sigma^i
\end{CD}\]
commutes, where $\pi_i\:I^{i+1}\to I^i$ is the projection along the last coordinate.
\end{proof}

We will say that a $\Delta$-space $X$ is {\it given} by an inverse sequence 
$\dots\xr{p_1}|K_1|\xr{p_0}|K_0|$
of simplicial maps between metric simplicial complexes $K_i$ if $X=\lim |K_i|$ as $\Delta$-spaces,
i.e.\ the simplexes and Hilbert simplexes of $X$ are inverse limits of simplexes of $K_i$.

\subsection{$\Delta$-subspaces}

By a {\it $\Delta$-subspace} of the $\Delta$-space $X$ we mean any closed subspace of $X$ which is 
a union of simplexes and Hilbert simplexes of $X$.

\begin{lemma} \label{delta-subspaces}
(a) The union $X^{(n)}$ of all $i$-simplexes of $X$ for all $i\le n$ is a $\Delta$-subspace.

(b) Every $\Delta$-subspace $Y$ of $X$ is itself a $\Delta$-space.
\end{lemma}

The $\Delta$-subspace of (a) will be called the {\it $n$-skeleton} of $X$.

\begin{proof}[Proof. (a)] If $X$ is given by an inverse sequence $\dots\xr{p_1}|K_1|\xr{p_0}|K_0|$, then
clearly $X^{(n)}=\lim |K_i^{(n)}|$, where $K_i^{(n)}$ denotes the usual $n$-skeleton of $K_i$.
\end{proof}

\begin{proof}[(b)] Let $X$ be given by an inverse sequence $\dots\xr{p_1}|K_1|\xr{p_0}|K_0|$.
Since $Y$ is a union of simplexes and Hilbert simplexes of $X$, each $p^\infty_i(Y)$ is a union 
of simplexes of $K_i$.
So $p^\infty_i(Y)=|L_i|$, where $L_i$ is the smallest subcomplex of $K_i$ such that $|L_i|$ 
contains $p^\infty_i(Y)$.
Then $\lim|L_i|$ certainly contains $Y$.
Given a thread $x=(x_0,x_1,\dots)\in\lim|L_i|$, each $x_i$ is the image of some $y_i\in Y$.
Then $y_i\to x$, so since $Y$ is closed, we get that $x\in Y$.
Thus $Y=\lim |L_i|$.
\end{proof}

\subsection{Countably dimensional $\Delta$-spaces}

A space $X$ is called {\it (strongly) countably dimensional} if $X$ is a union of countably many of its 
finite-dimensional (closed) subspaces $X_1,X_2,\dots$.
Here we may assume without loss of generality that $X_1\subset X_2\subset\dots$ (by replacing each $X_i$ with 
$X_1\cup\dots\cup X_i$).

Every subspace of a strongly countably dimensional metrizable space is itself strongly 
countably dimensional (see \cite{Sak2}*{Thoerem 5.3.3}).
In particular, subsets of polyhedra are strongly countably dimensional.
On the other hand, every strongly countably dimensional separable metrizable space embeds in a separable 
polyhedron, namely, in the space $c_{00}$ of finite sequences of reals, also known as $l_\infty^f$ 
(see \cite{M00}*{Theorem \ref{book:sfd-embedding}}).

\begin{proposition} The following are equivalent for a $\Delta$-space $X$:
\begin{enumerate}
\item $X$ is countably dimensional;
\item $X$ is strongly countably dimensional;
\item $X$ is the union of its finite-dimensional skeleta;
\item $X$ contains no Hilbert simplexes.
\end{enumerate}
\end{proposition}

\begin{proof} Clearly (4)$\Rightarrow$(3)$\Rightarrow$(2)$\Rightarrow$(1).
Suppose that $X$ is countably dimensional but contains a Hilbert simplex $\sigma^\infty$. 
Then $X$ is a union of its subspaces $X_1,X_2,\dots$ with $\dim X_i\le i$.
Also by Lemma \ref{infty-simplex} $\sigma^\infty$ contains a copy of the Hilbert cube $I^\infty$.
Now $I^\infty=\bigcup_i Y_i$, where each $Y_i=X_i\cap I^\infty$.
Each $Y_i$, being a closed subset of $X_i$, is easily seen to be of dimension $\le i$.
So we get that $I^\infty$ is countably dimensional.
But in reality $I^\infty$ is not countably dimensional 
(see \cite{Nag}*{VI.1.B} or \cite{Sak2}*{5.6.1 and 5.6.2}), so we obtain a contradiction.
\end{proof}

\subsection{Locally finite-dimensional $\Delta$-spaces}

A space $X$ is called {\it locally finite-dimensional} if every $x\in X$ has 
a finite-dimensional neighborhood.

\begin{lemma} A metrizable space $X$ is locally finite dimensional if and only if 
$X=\bigcup_{n=1}^\infty X_n$, where each $X_n$ is a closed $n$-dimensional subset of $X$ such that 
$X_n\subset\Int X_{n+1}$.
\end{lemma}

\begin{proof} 
Assuming that $X$ is finite dimensional, it admits an open cover $C$ whose elements are 
finite-dimensional.
Let $C_n$ be the set of all elements of $C$ that are of dimension $\le n$.
Let $Y_n$ be the union of all elements of $C_n$.
Each $x\in Y_n$ lies in some $U\in C_n$ together with some closed ball neighborhood $B_x$.
Since $B_x$ is a closed subset of $U$, it is easily seen to be of dimension $\le n$.
Hence $\dim Y_n\le n$ (see \cite{Sak2}*{5.4.4}).
We have $Y_1\subset Y_2\subset\dots$ and $X=\bigcup_n Y_n$, where each $Y_n$ is open.
Hence there also exist the desired $X_n$'s (see \cite{M00}*{Lemma \ref{book:chain-decomposition}(a)}).
\end{proof}

\begin{example} Let $K=\Delta^0\sqcup\Delta^1\sqcup\Delta^2\sqcup\dots$, the disjoint union of 
simplexes of increasing dimensions, regarded as a simplicial complex.
Let $Q=K\cup v*L$, where $L=\Delta^0\sqcup\Delta^0\sqcup\Delta^0\sqcup\dots\subset X$.
Clearly, $|Q|$ is locally finite dimensional.
However, $|\st(v,Q)|$ has no finite-dimensional neighborhood in $|Q|$.
\end{example}

\begin{proposition} \label{lfd}
Let $X$ be a $\Delta$-space given by an inverse sequence $\dots\xr{p_1}|K_1|\xr{p_0}|K_0|$ 
with surjective bonding maps.
The following are equivalent:
\begin{enumerate} 
\item $X$ is locally finite-dimensional;
\item $X$ is a union of its finite-dimensional subsets $X_0\subset X_1\subset\dots$ such that 
each $X\but X_n$ is a $\Delta$-subspace of $X$; 
\item $X$ is a union of its finite-dimensional subsets $X_n\bydef (p^\infty_n)^{-1}(|K_n|\but|L_n|)$, 
where $L_n$ are subcomplexes of $K_n$ satisfying $|L_{n+1}|\subset p_n^{-1}(|L_n|)$.
\item $X$ is a union of its finite-dimensional subsets of the form $(p^\infty_n)^{-1}(|K_n|\but|L|)$, where 
$|L|$ is a subcomplex of $|K_n|$.
\end{enumerate}
\end{proposition}

\begin{proof} Clearly (3)$\Rightarrow$(2)$\Rightarrow$(1).

Let us prove (1)$\Rightarrow$(4).
Clearly, $X$ contains no Hilbert simpexes.
Let $\sigma$ be a simplex of $X$.
Let $v$ be a vertex of $\sigma$ and let $U$ be its finite-dimensional neighborhood.
Then $U$ contains a neighborhood of the form $(p^\infty_n)^{-1}(V)$, where $V$ is a neighborhood of 
$v_n\bydef p^\infty_n(v)$ in $|K_n|$ for some $n$.
In turn, $V$ contains a neighborhood of the form $W_r\bydef h^{v_n}_r\big(\ost(v_n,\,K_n)\big)$ 
(see \cite{M00}*{Theorem \ref{book:open star}}).
By decreasing $r$ we may assume that $V$ contains the closure $\bar W_r$ of $W_r$.
Then $(p^\infty_n)^{-1}(\bar W_r)$ is a closed subset of $U$, so it is finite-dimensional.
Since the bonding maps are simplicial, it follows that $(p^\infty_n)^{-1}\big(\ost(v_n,\,K_n)\big)$
is also of some finite dimension $d$.
Since the bonding maps are surjective, $\dim (p^i_n)^{-1}\big(\ost(v_n,\,K_n)\big)\le d$ 
for all finite $i\ge n$.
This applies to each vertex $v$ of $\sigma$; since there are only finitely many of them, 
there is a single $n$ and a single $d$ that work for all of them.
Then $\dim (p^i_n)^{-1}\big(\bigcup_{v\in\sigma}\ost(v_n,\,K_n)\big)\le d$.
But $\bigcup_{v\in\sigma}\ost(v_n,\,K_n)=|K_n|\but |L|$, where $L$ is the subcomplex of $K_n$
consisting of all its simplexes that are disjoint from $\sigma_n\bydef p^\infty_n(\sigma)$.
Since $\sigma_n\subset|K_n|\but |L|$, we have $\sigma\subset (p^\infty_n)^{-1}(|K_n|\but|L|)$.

Let us prove (4)$\Rightarrow$(3).
Let $C_n$ be the set of all subsets of $X$ that are of dimension $\le n$
and of the form $(p^\infty_n)^{-1}(|K_n|\but |L|)$ for some subcomplex $L$ of $K_n$.
Then $C_0\subset C_1\subset\dots$ and $\bigcup_n C_n$ is a cover of $X$.
Since the bonding maps are surjective, $\dim(p^\infty_n)^{-1}(|K_n|\but |L|)\le n$ 
if and only if $\dim(p^i_n)^{-1}(|K_n|\but |L|)\le n$ for each finite $i\ge n$.
Therefore the union $X_n$ of all elements of $C_n$ is of dimension $\le n$.
Also $X_n=(p^\infty_n)^{-1}(|K_n|\but |L_n|)$, where $L_n$ is a subcomplex of $K_n$, 
namely, the intersection of all $L$'s as above.
Clearly, $p_n^{-1}(|K_n|\but |L_n|)\subset |K_{n+1}|\but |L_{n+1}|$.
\end{proof}

\section{Special $\Delta$-spaces} \label{special}

\subsection{$\nabla$-spaces}

We recall that a {\it $\nabla$-space} is the limit of an inverse sequence of non-degenerate
simplicial maps between metric simplicial complexes.

\begin{example} $\nabla$-spaces include all $0$-dim\-ensional metrizable spaces and the 
$p$-adic solenoid.
Let us note that the Cantor set admits very different choices of the structure of a $\Delta$-space 
(``$p$-adic'' for different $p$ and more generally ``$\ell$-adic'').

In fact, these examples of $\nabla$-spaces are rather special: they all are inverse limits of
simplicial immersions%
\footnote{A continuous map is called a (topological) {\it immersion} if it embeds some neighborhood 
of every point.}
between metric simplicial complexes.
\end{example}

\begin{example}
(a) The ``topologist's sine curve'' $X$ is a $\nabla$-space.
Indeed, $X$ is clearly the limit of the inverse sequence $\dots\xr{f_1}|J_1|\xr{f_0}|J_0|$ of 
simplicial maps between polygonal arcs $J_i$ with $i$ segments, where each $f_i$ is the retraction 
of $J_{i+1}$ onto $J_i$ sending the $(i+1)$st segment $I_{i+1}$ onto the $i$th one.

(b) The ``comb and flea'' space $\{\frac1n\mid n\in\N\}\x[0,1]\cup(0,1]\x\{1\}\cup\{(0,0)\}$ is 
a $\nabla$-space.
Indeed, $X$ is clearly the limit of the inverse sequence $\dots\xr{f_1}|G_1|\xr{f_0}|G_0|$ of 
simplicial maps between the graphs $G_n=[0,\infty)\x\{0\}\cup(\N\x[0,1])/(\{n,n+1,\dots\}\x\{1\})$, 
where the quotient is understood combinatorially (its topology is that of the metric simplicial complex, 
and not the quotient topology).
\end{example}

\begin{example}
Another notable example of a $\nabla$-space is the classifying space $B\Z_p$ of the $p$-adic integers, 
given by an inverse sequence of ramified coverings between the infinite-dimensional lens spaces 
$\dots\to L^\infty(p^2)\to L^\infty(p)$.
Namely, let $\Z/p^k$ act freely on $S^1$ in the usual way and diagonally on 
$S^\infty=\bigcup_i(\underbrace{S^1*\dots*S^1}_i)$. 
Then $L^\infty(p^k)$ is the orbit space of the latter action, and it is easy to see that the ramified covering
$S^\infty\to S^\infty$, defined as the join of the $p$-fold covers $S^1\to S^1$, descends to
a ramified covering $L^\infty(p^{k+1})\to L^\infty(p^k)$. 
\end{example}

\begin{example} $[0,1]$ with vertices at $1,\frac12,\frac14,\dots$ and $0$, and with edges of the form
$[2^{-i-1},2^{-i}]$ is a $\Delta$-space which is not a $\nabla$-space.
It is the inverse limit of the retractions $\dots\to[0,3]\to[0,2]\to[0,1]$ collapsing 
the leftmost edge $[i-1,i]$ onto its right endpoint $\{i-1\}$.
\end{example}

\begin{example} \label{Hawaiian}
The $n$-dimensional Hawaiian earring space $H_n$, i.e.\ the metric wedge $\bigvee_{i=1}^\infty S^n$
of copies of $S^n$ (see \cite{M00}*{\S\ref{book:metric wedge}}), is a $\Delta$-space which is not a $\nabla$-space. 
It is the inverse limit of the retractions
$\dots\to S^n\vee S^n\vee S^n\to S^n\vee S^n\to S^n$ between the finite wedges.
\end{example}

\subsection{Normal $\Delta$-subspaces}

\begin{lemma} \label{inverse subdivision}
Let $\dots\xr{p_1}K_1\xr{p_0}K_0$ be an inverse sequence of simplicial maps between simplicial complexes.
Then their exist admissible subdivisions $K_i'$ of $K_i$, isomorphic to their barycentric subdivisions
and such that each $p_i\:K_{i+1}'\to K_i'$ is simplicial.
Moreover, $\lim|K_i'|$ is homeomorphic to $\lim|K_i|$ via the identity map.
\end{lemma}

\begin{proof}
If each $p_i$ is non-degenerate, then it sends barycenters to barycenters, and in this case we may define
$K_i'$ to be the barycentric subdivision $K_i^\flat$ for all $i$, which is admissible 
(see \cite{M00}*{Lemma \ref{book:barycentric-admissible}}).
In the general case, let $K_0^+=K_0$, viewed as an apexed simplicial complex with apexes of simplexes 
chosen to be at their barycenters.
Assuming that $K_i^+$ is defined, let $K_{i+1}^+$ be an apexed simplicial complex such that 
$p_i\:K_{i+1}^+\to K_i^+$ is a conical map (see \cite{M00}*{\S\ref{book:polytopal complexes}}).
Let $K_i'=(K_i^+)^\flat$, which is an admissible subdivision of $K_i$ 
\cite{M00}*{Lemma \ref{book:map-subdivision}}.
Then each $p_i\:K_{i+1}'\to K_i'$ is simplicial.

Since each $K'_i$ admissibly subdivides $K_i$, the map $\id\:|K_i'|\to|K_i|$ is a homeomorphism.
Hence the identity map between the inverse limits is also a homeomorphism.
\end{proof}

A $\Delta$-subspace $Y$ of a $\Delta$-space $X$ will be called {\it normal} if there exists a neighborhood $U$
of $Y$ such that if a simplex or Hilbert simplex of $X$ lies in $U$, then it lies in $Y$.

\begin{proposition} \label{normality}
Let $X$ be a $\Delta$-space given by an inverse sequence 
$\dots\xr{p_1}K_1\xr{p_0}K_0$.
Lemma \ref{inverse subdivision} yields an inverse sequence $\dots\xr{p_1}K_1'\xr{p_0}K_0'$
which represents some $\Delta$-space $X'$. (Thus $X'=X$ as spaces.)
The following are equivalent:
\begin{enumerate}
\item $Y$ is a normal $\Delta$-subspace of $X$;
\item there exists a neighborhood $U$ of $Y$ such that $U\but Y$ contains no vertices of $X'$;
\item $Y=\bigcup_n(p^\infty_n)^{-1}(|L_n|)$ for some subcomplexes $L_n$ of $K_n$.
\end{enumerate}
\end{proposition}

\begin{proof} (3)$\Rightarrow$(2).
Let $L_n^*$ be the subcomplex of $K_n'$ consisting of all simplexes that are disjoint from $|L_n|$.
Then $|L_n^*|$ is closed in $|K_n'|=|K_n|$.
Therefore $U_n\bydef (p^\infty_n)^{-1}(|K_n|\but|L_n^*|)$ is a neighborhood of $Y_n\bydef (p^\infty_n)^{-1}(|L_n|)$.
Hence $U\bydef \bigcup_n U_n$ is a neighborhood of $Y=\bigcup_n Y_n$.
If $v$ is a vertex of $X'$ that is contained in $U$, then $v$ must lie in some $U_n$.
Hence $v_n\bydef p^\infty_n(v)$ is a vertex of $K_n'$ not contained in $|L_n^*|$.
Then $v_n\in |L_n|$.
Hence $v\in Y_n\subset Y$.

(2)$\Rightarrow$(1).
It suffices to show that every simplex $\sigma$ of $X$ contains a vertex of $X'$.
Indeed, let $v_i$ be the apex of $\sigma_i\bydef p^\infty_i(\sigma)$ in the apexed simplicial complex
$K_i^+$ (see the proof of Lemma \ref{inverse subdivision}).
Then each $p_i(v_{i+1})=v_i$, and so $v\bydef (v_0,v_1,\dots)$ is a vertex of $X'$.

(1)$\Rightarrow$(3).
Suppose that $U$ is a neighborhood of $Y$ in $X$ which contains no simplexes and
Hilbert simplexes of $X$ other than those of $Y$.
We may assume that $U$ is open.
Then $U$ is a union of sets of the form $(p^\infty_n)^{-1}(V)$, where $V$ is an open subset 
of $|K_i|$.
For each $n$ there may be many such $V$'s, but by taking their union we may assume that 
there is only one; let us denote it $V_n$.
Thus $U=\bigcup_n U_n$, where $U_n=(p^\infty_n)^{-1}(V_n)$.
By replacing each $V_n$ with $(p^n_0)^{-1}(V_0)\cup(p^n_1)^{-1}(V_1)\cup\dots\cup V_n$ 
we may assume that each $p_n^{-1}(V_n)\subset V_{n+1}$.
Hence also $U_0\subset U_1\subset\dots$.
Let $L_n$ be the subcomplex of $K_n$ consisting of all its simplexes that lie in $V_n$.
Thus $p_n^{-1}(|L_n|)\subset |L_{n+1}|$.
Let $Z=\bigcup_n(p^\infty_n)^{-1}(|L_n|)$.
Then $Z$ is a $\Delta$-subspace of $X$ which lies in $U$, so by our hypothesis every 
simplex of $Z$ lies in $Y$.
Thus $Z\subset Y$.
On the other hand, if $\sigma$ is a simplex of $Y$, then, being compact, $\sigma$ 
must lie in some $U_n$.
Then $\sigma_n\bydef p^\infty_n(\sigma)$ lies in $V_n$, and hence, being a simplex of $K_n$,
also in $|L_n|$.
Therefore $\sigma$ lies in $Z$.
Thus $Y\subset Z$.
\end{proof}

\subsection{Decomposable $\Delta$-spaces}

Clearly, the finite-dimensional skeleta of a $\nabla$-space are its normal $\Delta$-subspaces.
More generally, we will call a $\Delta$-space {\it decomposable} if it is a union of its finite-dimensional 
normal $\Delta$-subspaces.

\begin{example} \label{Hawaiian2}
Let $H_n$ be the $n$-dimensional Hawaiian earring space (see Example \ref{Hawaiian2}), 
but now regarded as a subset of $|\Delta^{n+1}|$.
Let us fix any sequence of integers $n_1,n_2,\dots$, and let 
$X=\bigcup_{i=1}^\infty H_{n_1}*\dots*H_{n_i}$, regarded as a subspace of the metric simplicial complex
$|\bigcup_{i=1}^\infty\Delta^{n_1+1}*\dots*\Delta^{n_i+1}|$.
Then $X$ is a decomposable $\Delta$-space which is not a $\nabla$-space.
\end{example}

\begin{example}
(a) The null-sequence $X=\nullseq_{n=1}^\infty\Delta^n$ of simplexes of increasing dimensions 
(see \cite{M00}*{\S\ref{book:metric wedge}}) is a strongly countably dimensional but non-decomposable
$\Delta$-space.
It is the inverse limit of the retractions
$\dots\to pt\sqcup\Delta^3\sqcup\Delta^2\sqcup\Delta^1\to pt\sqcup\Delta^2\sqcup\Delta^1\to pt\sqcup\Delta^1$
sending the top-dimensional simplex onto $pt$.

(b) The metric cone $CX$ (see \cite{M00}*{\S\ref{book:metric cone}}), where $X$ is the null-sequence of (a), 
is another example of a strongly countably dimensional but non-decomposable $\Delta$-space.
Let us note that the barycenters of the maximal simplexes of $CX$ converge to its vertex $pt$.
\end{example}

\begin{proposition} \label{decomposable} 
Let $X$ be a $\Delta$-space given by an inverse sequence $\dots\xr{p_1}K_1\xr{p_0}K_0$ 
with surjective bonding maps.
The following are equivalent:
\begin{enumerate}
\item $X$ is decomposable;
\item $X$ is a union of an increasing sequence $X_0\subset X_1\subset\dots$ of its 
finite-dimensional normal $\Delta$-subspaces;
\item $X$ is a union of its finite-dimensional subspaces $X_n\bydef (p^\infty_n)^{-1}(|L_n|)$, 
where $L_n$ are subcomplexes of $K_n$ satisfying $p_n^{-1}(|L_n|)\subset |L_{n+1}|$.
\end{enumerate}
\end{proposition}

\begin{proof} Clearly (3)$\Rightarrow$(2)$\Rightarrow$(1).
Let us prove (1)$\Rightarrow$(3).
By Proposition \ref{normality} $X$ is a union of its finite-dimensional subspaces 
that are of the form $(p^\infty_n)^{-1}(|L|)$ for some subcomplex $L$ of $K_n$
(using that a closed subset of a finite dimensional space is finite dimensional).
The remainder of the proof is similar to the proof of the corresponding implication 
in Proposition \ref{lfd} (using union rather than intersection of subcomplexes).
\end{proof}

\subsection{Telescopic resolution}

\begin{proposition} \label{cofibration}
Let $X$ be a $\Delta$-space given by an inverse sequence $\dots\xr{p_1}K_1\xr{p_0}K_0$.
Let $Y=(p^\infty_n)(L_n)$ for some subcomplex $L_n$ of some $K_n$.
Then the inclusion $Y\to X$ is a cofibration.
\end{proposition}

\begin{proof} 
Let $L_i=(p^i_n)^{-1}(|L_n|)$ for each $i>n$.
Let $K_i'$ and $K_i''$ be given by Lemma \ref{inverse subdivision}.
Thus each $p_i$ yields simplicial maps $K_{i+1}'\to K_i'$ and $K_{i+1}''\to K_i''$.

Let $N_i$ be the subcomplex of $K_i''$ consisting of all simplexes that meet $|L_i|$ and of their faces.
Then there is a strong deformation retraction of $|N_i|$ onto $|L_i|$.
Namely, a retraction $r_i\:|N_i|\to|L_i|$ is defined by sending the apex of a simplex 
$\sigma\in K_i'$ which meets $|L_i|$ onto the apex of $\sigma\cap |L_i|$, which is a
simplex of $L_i'$ (see details in \cite{M00}*{Proof of Proposition \ref{book:nbd-retract}}).
A homotopy $h_i$ between $r_i$ and $\id_{|N_i|}$ is defined by $h_i(x,t)=(1-t)x+tr_i(x)$,
using that $x$ and $r_i(x)$ always belong to the same simplex of $K_i'$
(compare \cite{M00}*{Proof of Corollary \ref{book:mcn-nbd}}).

Since $|L_i|=p_{i-1}^{-1}(|L_{i-1}|)$, we have $|N_i|=p_{i-1}^{-1}(|N_{i-1}|)$.
Also, if $\sigma\in K_{i+1}'$ then $p_i(\sigma\cap|L_{i+1}|)=p_i(\sigma)\cap|L_i|$ due to 
$|L_i|=p_{i-1}^{-1}(|L_{i-1}|)$.
Hence the retractions $r_i$ and the homotopies $h_i$ commute with the bonding maps.
Therefore $N\bydef \lim |N_i|$ strong deformation retracts onto $Y$.
Since each $|N_i|=(p^i_n)^{-1}(|N_n|)$, we have $N=(p^\infty_n)^{-1}(|N_n|)$.
Hence $N$ is a neighborhood of $Y$ in $X$.
Thus the inclusion $Y\to X$ is a cofibration (see \cite{M00}*{Lemma \ref{book:cof3}}).
\end{proof}

From Proposition \ref{cofibration}, Proposition \ref{decomposable} and 
\cite{M00}*{Proposition \ref{book:lc-telescope-b}} we obtain

\begin{theorem} \label{telescope-decomposition}
(a) Every decomposable $\Delta$-space is homotopy equivalent to the mapping telescope of its 
finite-dimensional normal $\Delta$-subspaces.

(b) Every $\nabla$-space is homotopy equivalent to the mapping telescope of its finite dimensional skeleta.
\end{theorem}

\begin{corollary} \label{lfd-nabla}
(a) Every decomposable $\Delta$-space is homotopy equivalent to a locally finite-dimensional $\Delta$-space.

(b) Every $\nabla$-space is homotopy equivalent to a locally finite-dimensional $\nabla$-space.
\end{corollary}

\begin{proposition} Every locally finite-dimensional $\Delta$-space is decomposable.
\end{proposition}

Let us note that the $\Delta$-space $\bigcup_{i=1}^\infty H_{n_1}*\dots*H_{n_i}$ of Example \ref{Hawaiian2}
is decomposable, but not locally finite-dimensional.

\begin{proof} Let $X$ be given by an inverse sequence $\dots\xr{p_1}|K_1|\xr{p_0}|K_0|$.
By Proposition \ref{lfd} every simplex $\sigma$ of $X$, being compact, lies in
a finite-dimensional subspace of the form $(p^\infty_n)^{-1}(|K_n|\but|L|)$, where 
$|L|$ is a subcomplex of $|K_n|$.
Then $\sigma_n\bydef p^\infty_n(\sigma)$ is disjoint from $|L|$.
Hence $(p^\infty_n)^{-1}(\sigma_n)$ is a finite-dimensional normal $\Delta$-subspace of $X$
containing $\sigma$.
\end{proof}

\section{$\Delta$-resolution}

A compactum is called {\it cell-like} it it has the fine shape of a point; a perfect map
is called {\it cell-like} if every its point-inverse is a cell-like compactum.
Every cell-like map between finite dimensional compacta is a fine shape equivalence 
(see \cite{M1}*{Theorem 3.15(e)}).
However, a cell-like map between infinite-dimensional compacta may fail to be a fine shape 
equivalence and even a shape equivalence (see \cite{DS}*{10.3.1}, \cite{Sak3}).

\begin{theorem} \label{delta}
(a) Every Polish space $X$ is fine shape equivalent to a Polish $\Delta$-space $Y$.
Moreover:

(b) $Y$ is given by an inverse sequence $\dots\to|Q_1|\to|Q_0|$, where each $Q_i$ is 
finite dimensional, locally finite and countable;

(c) there is a continuous map $\xi\:Y\to X$ which is a fine shape equivalence;

(d) $\xi$ is a cell-like perfect map;

(e) if $X$ of dimension $\le n$, then so is $Y$;

(f) if $X$ is locally finite dimensional, then so is $Y$.
\end{theorem}

\subsection{The $\Delta$-space $Y$}

\begin{proof}[Proof of Theorem \ref{delta}(a,b)]
By Isbell's theorem (see \cite{M00}*{Theorem \ref{book:isbell}}) the given Polish space $X$ is the limit 
of a scalable inverse sequence $\dots\xr{p_1}|P_1|\xr{p_0}|P_0|$ of polyhedral maps between polyhedra $|P_i|$, 
where each simplicial complex $P_i$ is countable, locally finite and finite dimensional
and each $p_i$ is a simplicial map from $P_{i+1}$ to some admissible subdivision $P_i^*$ of $P_i$.
If $x\in |P_i|$, let $\Delta_x$ and $\Delta^*_x$ denote the smallest simplexes of $P_i$ and $P_i^*$ 
containing $x$.
Then $p_i(\Delta_x)\subset\Delta^*_{p_i(x)}\subset\Delta_{p_i(x)}$ for each $x\in |P_{i+1}|$.
Let $|P|_{[0,\infty]}$ denote the extended mapping telescope of the inverse sequence 
$\dots\xr{p_1}|P_1|\xr{p_0} |P_0|$ (see \cite{M00}*{\S\ref{book:extended}}).
Thus $X$ is a Z-set in the ANR $|P|_{[0,\infty]}$ and $|P|_{[0,\infty]}\but X$ is 
the infinite mapping telescope $|P|_{[0,\infty)}$.

By Sakai's theorem (see \cite{M00}*{Theorem \ref{book:simp-appr0}}) each bonding map $p_i\:|P_{i+1}|\to|P_i|$ 
is homotopic to a map $q_i\:|P_{i+1}|\to|P_i|$, which is a simplicial map from some admissible subdivision 
$P'_{i+1}$ of $P_{i+1}$ into $P_i$, by a homotopy $h^i_t\:|P_{i+1}|\to |P_i|$.
Moreover, by the proof of Sakai's theorem $q_i(\Delta_x)\subset\Delta_{p_i(x)}$ for each $x\in |P_{i+1}|$
and $h_t^i$ is defined by $h_t^i(x)\bydef (1-t)p_i(x)+tq_i(x)$ for each $x\in |P_{i+1}|$ and each $t\in I$.
Since $p_i(\Delta_x)\subset\Delta_{p_i(x)}$, we get that $h_t^i(\Delta_x)\subset\Delta_{p_i(x)}$ for each $t\in I$.
Consequently, we obtain the following property for any pair of subcomplexes $K_i\subset P_i$ and 
$K_{i+1}\subset P_{i+1}$ such that $p_i(|K_{i+1}|)\subset |K_i|$:
\[\begin{gathered}\label{property*}
h_t^i(|K_{i+1}|)\subset|K_i|\text{ for all $t\in I$;}\\
\text{ in particular, }q_i(|K_{i+1}|)\subset |K_i|.
\end{gathered}\tag*{($*$)}\]

Let $Q_0=P_0$ and $Q_1=P_1'$.
Assuming that an admissible subdivision $Q_i$ of $P_i$ has been defined, by 
\cite{M00}*{Theorem \ref{book:map subdivision}} there exists an admissible subdivision $Q_{i+1}$ of $P_{i+1}'$ 
such that $q_i$ is a simplicial map from $Q_{i+1}$ to $Q_i$.
Let $Y$ be the limit of the inverse sequence $\dots\xr{q_1}|Q_1|\xr{q_0}|Q_0|$ and let 
$|Q|_{[0,\infty]}=|Q|_{[0,\infty)}\cup Y$ denote its extended mapping telescope.
Then $Y$ is a $\Delta$-space.
Since each $Q_i$ is countable, $Y$ is separable, and since each $Q_i$ is locally finite, $Y$ is
completely metrizable; thus $Y$ is Polish.
If $y\in |Q_i|$, let $\Delta'_y$ denote the smallest simplex of $Q_i$ containing $y$.
Since $q_i\:Q_{i+1}\to Q_i$ is simplicial, $q_i(\Delta'_y)\subset\Delta'_{q_i(y)}$.

Since $p_i$ and $q_i$ are homotopic, their mapping cylinders are homotopy equivalent relative 
to $|P_i|\cup |P_{i+1}|$.
Namely, $f_i\:|P|_{[i,\,i+1]}\to |Q|_{[i,\,i+1]}$ is defined by sending $|P|_{[i+\frac12,\,i+1]}$ 
into $|Q|_{[i,\,i+1]}$ using the increasing linear homeomorphism $\ell_i\:[i+\frac12,\,i+1]\to[i,\,i+1]$ 
and $|P|_{[i,\,i+\frac12]}$ into $|Q_i|$ using $h_t^i$.
Similarly, $g_i\:|Q|_{[i,\,i+1]}\to |P|_{[i,\,i+1]}$ is defined by sending $|Q|_{[i+\frac12,\,i+1]}$ 
into $|P|_{[i,\,i+1]}$ using $\ell_i$ and $|Q|_{[i,\,i+\frac12]}$ into $|P_i|$ using $h_{1-t}^i$.
Then $g_if_i$ is homotopic to $\id_{|P|_{[i,\,i+1]}}$ by the obvious homotopy 
$\phi_i\:|P|_{[i,\,i+1]}\x I\to |P|_{[i,\,i+1]}$ and $f_ig_i$ is homotopic to $\id_{|Q|_{[i,\,i+1]}}$ 
by the obvious homotopy $\psi_i\:|P|_{[i,\,i+1]}\x I\to |P|_{[i,\,i+1]}$.
Consequently, $|P|_{[0,\infty)}$ and $|Q|_{[0,\infty)}$ are homotopy equivalent via some maps 
$F\:|P|_{[0,\infty)}\to |Q|_{[0,\infty)}$ and $G\:|Q|_{[0,\infty)}\to |P|_{[0,\infty)}$ such that $GF$ 
is homotopic to $\id_{|P|_{[0,\infty]}}$ via some homotopy $\Phi\:|P|_{[0,\infty)}\x I\to |P|_{[0,\infty)}$ 
and $FG$ is homotopic to $\id_{|Q|_{[0,\infty]}}$ via some homotopy 
$\Psi\:|Q|_{[0,\infty)}\x I\to |Q|_{[0,\infty)}$.

The constructed maps and homotopies satisfy the following property for any family of subcomplexes 
$K_i\subset P_i$ such that each $p_i(|K_{i+1}|)\subset |K_i|$.
Let $|K|_{[0,\infty)}$ be the mapping telescope of the inverse sequence 
$\dots\xr{p_1|}|K_1|\xr{p_0|}|K_0|$.
By \ref{property*} we also have the mapping telescope
$|L|_{[0,\infty)}$ of the inverse sequence $\dots\xr{q_1|}|L_1|\xr{q_0|}|L_0|$, where each $L_i$ 
is the subcomplex of $Q_i$ subdividing $K_i$.
Then it follows from \ref{property*} that each $f_i(|K|_{[i,\,i+1]})\subset |L|_{[i,\,i+1]}$ and 
each $g_i(|L|_{[i,\,i+1]})\subset |K|_{[i,\,i+1]}$; moreover, 
each $\phi_i(|K|_{[i,\,i+1]}\x I)\subset |K|_{[i,\,i+1]}$ and 
each $\psi_i(|L|_{[i,\,i+1]}\x I)\subset |L|_{[i,\,i+1]}$.
Consequently, 
\[\begin{gathered}\label{property**}
F(|K|_{[0,\infty)})\subset |L|_{[0,\infty)}\text{ and }
G(|L|_{[0,\infty)})\subset |K|_{[0,\infty)};\text{ moreover,}\\
\Phi(|K|_{[0,\infty)}\x I)\subset |K|_{[0,\infty)}\text{ and }
\Psi(|L|_{[0,\infty)}\x I)\subset |L|_{[0,\infty)}.
\end{gathered}\tag*{($**$)}\]

\begin{lemma} \label{convergence}
If $x_k$ is a sequence of points in $|P|_{[0,\infty)}$ converging to an $x_\infty\in X$,
then $M\bydef \{x_k\mid k\in\N\}$ lies in the mapping telescope $|K|_{[0,\infty)}$ of an inverse sequence
of the form $\dots\xr{p_1|}|K_1|\xr{p_0|}|K_0|$, where $K_i\subset P_i$ are finite subcomplexes 
such that each $p_i(|K_{i+1}|)\subset |K_i|$.
In fact, $\lim|K_i|=\{x_\infty\}$.
\end{lemma}

\begin{proof}
Each $|P|_{(i,\,i+1]}$ contains only finitely many of the $x_k$, and consequently $|P_{i+1}|$ 
contains only finitely many of their images $x'_k$ under the composition 
$|P|_{(i,\,i+1]}\xr{\cong} |P_{i+1}|\x (0,1]\xr{\text{projection}} |P_{i+1}|$.
Then the $x'_k$ converge to $x_\infty$, and hence $M'_+\bydef \{x_\infty\}\cup\{x'_k\mid k\in\N\}$ 
is compact.
Then the image $M_i$ of $M'_+\cap |P|_{[i,\infty]}$ under the retraction 
$\Pi_i\:|P|_{[i,\infty]}\to |P_i|$ is compact.
Let $K_i$ be the smallest subcomplex of $P_i$ such that $M_i\subset|K_i|$.
Since $M_i$ is compact and $P_i$ is locally finite, $K_i$ is finite.
Since $p_i(M_{i+1})\subset M_i$ and $p_i(\Delta_x)\subset\Delta_{p_i(x)}$ for each 
$x\in |P_{i+1}|$, we have $p_i(|K_{i+1}|)\subset|K_i|$.
If $x_k\in |P|_{(i,\,i+1]}$, then $x'_k\in M_i$ and consequently $x_k\in |K|_{(i,\,i+1]}$.
Thus $M\subset |K|_{[0,\infty)}$.

It is easy to see that $\lim M_i=\{x_\infty\}$.
Since the inverse sequence $\dots\xr{p_1}|P_1|\xr{p_0}|P_0|$ is scalable, it follows that 
$\lim|K_i|=\{x_\infty\}$ (see \cite{M00}*{Proposition \ref{book:scalable}(b)}).
\end{proof}

Given a sequence $x_i$ of points in $|P|_{[0,\infty)}$ converging to an $x_\infty\in X$, 
Lemma \ref{convergence} yields a $|K|_{[0,\infty)}$ containing $M\bydef \{x_k\mid k\in\N\}$.
Also by \ref{property*} we get the extended mapping telescope $|L|_{[0,\infty]}$ of 
the inverse sequence $\dots\xr{q_1|}|L_1|\xr{q_0|}|L_0|$, where each $L_i$ is the subcomplex 
of $Q_i$ subdividing $K_i$.
Then $F(M)\subset F(|K|_{[0,\infty)})$ and by \ref{property**} also 
$F(|K|_{[0,\infty)})\subset |L|_{[0,\infty)}$, where $|L|_{[0,\infty]}$ is compact.
Thus $F$ is an $X$-$Y$-approaching map.

If $(x_k,t_k)$ is a sequence of points in $|P|_{[0,\infty)}\x I$ converging to 
an $(x_\infty,t_\infty)\in X\x I$,
then $x_k\to x_\infty$, and hence $\hat M\bydef \{(x_k,t_k)\mid k\in\N\}$ lies in $|K|_{[0,\infty)}\x I$, 
where $|K|_{[0,\infty)}$ is yielded by Lemma \ref{convergence}.
Then $\Phi(\hat M)\subset\Phi(|K|_{[0,\infty)}\x I)$ and by \ref{property**} also 
$\Phi(|K|_{[0,\infty)}\x I)\subset |K|_{[0,\infty)}$, where $|K|_{[0,\infty]}$ is compact.
Thus $\Phi$ is an $X$-$X$-approaching homotopy.

\begin{lemma} \label{going back}
Suppose that $C_i\subset Q_i$ are subcomplexes such that each 
$q_i(|C_{i+1}|)\subset|C_i|$.
Let $K_i$ be the subcomplex of $P_i$ consisting of all simplexes that meet $|C_i|$ and of 
all their faces.
Then each $p_i(|K_{i+1}|)\subset|K_i|$.
\end{lemma}

Let us note that if $\tilde K_i$ is the smallest subcomplex of $P_i$ such that 
$|C_i|\subset|\tilde K_i|$, then $p_i(|\tilde K_{i+1})\not\subset|\tilde K_i|$ in general.

\begin{proof}
Let $\sigma$ be a maximal simplex of $K_{i+1}$.
Since $Q_{i+1}$ is a subdivision of $P_{i+1}$, $\sigma$ contains a simplex $\sigma'$ 
of $C_{i+1}$.
Let $\tau$ be the smallest simplex of $P_i$ containing $p_i(\sigma)$.
Then by \ref{property*} $q_i(\sigma)\subset\tau$.
Then also $q_i(\sigma')\subset\tau$.
Since $q_i(\sigma')\in C_i$, we get that $\tau\in K_i$.
Then also $p_i(\sigma)\subset|K_i|$.
\end{proof}

\begin{lemma} \label{convergence2}
If $y_k$ is a sequence of points in $|Q|_{[0,\infty)}$ converging to a $y_\infty\in Y$,
then $N\bydef \{y_k\mid k\in\N\}$ lies in the mapping telescope $|L|_{[0,\infty)}$ of an inverse sequence
$\dots\xr{q_1|}|L_1|\xr{q_0|}|L_0|$, where $L_i\subset Q_i$ are subcomplexes subdividing
finite subcomplexes $K_i\subset P_i$ such that each $p_i(|K_{i+1}|)\subset |K_i|$ and consequently 
by \ref{property*} also each $q_i(|L_{i+1}|)\subset |L_i|$.
\end{lemma}

\begin{proof}
Since $q_i(\Delta'_y)\subset\Delta'_{q_i(y)}$ for each $y\in |Q_{i+1}|$, the proof of 
Lemma \ref{convergence} works to show that $N$ lies in the mapping telescope $|C|_{[0,\infty)}$ 
of an inverse sequence $\dots\xr{p_1|}|C_1|\xr{p_0|}|C_0|$, where $C_i\subset Q_i$ are 
finite subcomplexes such that each $q_i(|C_{i+1}|)\subset|C_i|$.
Let $K_i$ be the subcomplex of $P_i$ consisting of all simplexes that meet $|C_i|$ and of 
all their faces.
Then by Lemma \ref{going back} each $p_i(|K_{i+1}|)\subset|K_i|$.
\end{proof}

Given a sequence $y_i$ of points in $|Q|_{[0,\infty)}$ converging to a $y_\infty\in Y$, 
Lemma \ref{convergence} yields an $|L|_{[0,\infty)}$ containing $N\bydef \{y_k\mid k\in\N\}$,
and also the extended mapping telescope $|K|_{[0,\infty]}$ of the inverse sequence 
$\dots\xr{p_1|}|K_1|\xr{p_0|}|K_0|$.
Then $G(N)\subset G(|L|_{[0,\infty)})$ and by \ref{property**} also
$G(|L|_{[0,\infty)})\subset |K|_{[0,\infty)}$, where $|K|_{[0,\infty]}$ is compact.
Thus $G$ is an $X$-$Y$-approaching map.

If $(y_k,t_k)$ is a sequence of points in $|Q|_{[0,\infty)}\x I$ converging to a 
$(y_\infty,t_\infty)\in Y\x I$, then $y_k\to y$, and hence 
$\hat N\bydef \{(y_k,t_k)\mid k\in\N\}$ lies in $|L|_{[0,\infty)}\x I$,
where $|L|_{[0,\infty)}$ is yielded by Lemma \ref{convergence2}.
Then $\Psi(\hat N)\subset\Psi(|L|_{[0,\infty)}\x I)$ and by \ref{property**} also
$\Psi(|L|_{[0,\infty)}\x I)\subset |L|_{[0,\infty)}$, where $|L|_{[0,\infty]}$ is compact.
Thus $\Psi$ is a $Y$-$Y$-approaching homotopy.
\end{proof}

\subsection{The map $\xi$}

\begin{proof}[Proof of Theorem \ref{delta}(c)] 
Whether Isbell's theorem is proved based on the closed embedding in $\R^\infty$
(see \cite{M00}*{Theorem \ref{book:isbell}}) or on the nerves of covers 
(see \cite{M00}*{Remark \ref{book:nerves and cubes}(0,2)}, this can be done uniformly; 
that is, with appropriate metrics on the $|P_i|$ and on $X$, the latter is the limit of 
the former ones in the uniform category (see \cite{M3}*{Theorem \ref{poly:inverse limit}} 
for the details of the second approach).
Moreover, by construction, each simplex of $P_i$ is of diameter $\le 1$, and each
$p^j_i\:|P_j|\to |P_i|$, where $j\ge i$, sends any pair of points at a distance $\le 1$ 
into a pair of points at a distance $\le 2^{i-j}$.
Since each $P_i$ is locally finite, each $|P_i|$ is complete.

Given any $x\in |P_{i+1}|$, we have $p_i(\Delta_x)\subset\Delta_{p_i(x)}$, and hence 
by \ref{property*} also $q_i(\Delta_x)\subset\Delta_{p_i(x)}$.
Since $\Delta_{p_i(x)}$ is of diameter $\le 1$, the distance between $p_i(x)$ and 
$q_i(x)$ is $\le 1$.

Using the standard metric on each $|Q_i|$ (determined by the simplicial complex $Q_i$,
not by $P_i$), the simplicial maps $q_i\:|Q_{i+1}|\to |Q_i|$ are uniformly continuous, 
and this also yields a metric on $Y$ such that it is the inverse limit of the $|Q_i|$ 
in the uniform category.
Since each $Q_i$ is a subdivision of $P_i$, the map $\id\:|Q_i|\to|P_i|$ is uniformly
continuous (this does not use the admissibility of the subdivision).

Then by applying \cite{M2}*{Proposition \ref{metr:A.15}(a)} with $\alpha_i=\beta_i=1$ 
to the maps $\id\:|Q_i|\to |P_i|$ we get a uniformly continuous map $\xi\:Y\to X$
which combines with the maps $\id\:|Q_i|\to |P_i|$ into a uniformly continuous 
level-preserving map $|Q|_{\N^+}\to|P|_{\N^+}$, where $\N^+$ denotes the subset 
$\N\cup\{\infty\}$ of $[0,\infty]$.
It is not hard to see that the latter in turn combines with the map 
$G\:|Q|_{[0,\infty)}\to |P|_{[0,\infty)}$ constructed in the proof of (a) 
into a uniformly continuolus map $\Xi\:|Q|_{[0,\infty]}\to |P|_{[0,\infty]}$.
Hence the fine shape class of $G$ is induced by the homotopy class of $\xi$.
\end{proof}

\begin{proof}[Proof of Theorem \ref{delta}(d)] It is easy to see that the map $\xi\:Y\to X$ of (c)
is defined by $\xi(\tau)=x_\tau$, where $x_\tau$ is given by the following lemma.

\begin{lemma} \label{kaul}
Let $\tau$ be a simplex or Hilbert simplex of the $\Delta$-space $Y$ and let 
$\tau_i=p^\infty_i(\tau)$.
Then each $\tau_i$ lies in a simplex $\sigma_i$ of $P_i$ such that each 
$p_i(\sigma_{i+1})\subset\sigma_i$.
Moreover, $\lim\sigma_i$ consists of a single point $x_\tau\in X$.
\end{lemma}

Let us note that the points $p^\infty_i(x_\tau)$ may well lie in the boundaries of 
the simplexes $\sigma_i$ and not in their interiors.

\begin{proof}
Since $Q_i$ is a subdivision of $P_i$, some simplex of $P_i$ contains $\tau_i$.
Let $\sigma_{ii}$ be the smallest such simplex.
Let us fix some $k\ge i$.
Since each $p_i$ sends each simplex of $P_{i+1}$ into some simplex of $P_i$, some 
simplex of $P_i$ contains $p^k_i(\sigma_{kk})$ and consequently also $p^k_i(\tau_k)$.
Let $\sigma_{ik}$ be the smallest simplex of $P_i$ containing $p^k_i(\tau_k)$.
Then $p_i(\sigma_{i+1,\,k})\subset\sigma_{ik}$ by the minimality, and consequently 
by \ref{property*} $q_i(\sigma_{i+1,\,k})\subset\sigma_{ik}$.
Hence $\tau_i=q_i(\tau_{i+1})\subset q_i(\sigma_{i+1,\,i+1})\subset\sigma_{i,\,i+1}$.
Thus $\tau_k\subset\sigma_{k,\,k+1}$.
Also $p^k_i(\sigma_{k,\,k+1})\subset\sigma_{i,\,k+1}$ by the minimality.
Hence $p^k_i(\tau_k)\subset p^k_i(\sigma_{k,\,k+1})\subset\sigma_{i,\,k+1}$.
Consequently $\sigma_{ik}\subset\sigma_{i,\,k+1}$ by the minimality.

Since $P_i$ is finite-dimensional, the chain 
$\sigma_{ii}\subset\sigma_{i,\,i+1}\subset\sigma_{i,\,i+2}\subset\dots$
must stabilize at some simplex $\sigma_i=\colim_j\sigma_{ij}$.
Then $p^k_i(\sigma_k)\subset\sigma_i$.
Also $\tau_i\subset\sigma_{ii}\subset\sigma_i$.
Since the inverse sequence $\dots\xr{p_1}|P_1|\xr{p_0}|P_0|$ is scalable,
$\lim\sigma_i$ consists of a single point $x_\tau\in X$ 
(see \cite{M00}*{Proposition \ref{book:scalable}(b)}).
\end{proof}

Given a point $x=(x_0,x_1,\dots)\in X$, let $\sigma_i$ be the smallest simplex 
of $P_i$ containing $x_i$, and let $S_i=\st(\sigma_i,P_i)$.
Clearly $p_i(|S_{i+1}|)\subset |S_i|$, and hence by \ref{property*} also 
$q_i(|T_{i+1}|)\subset |T_i|$, where $T_i$ is the subcomplex of $Q_i$ subdividing $S_i$.
Let $T(x)=\lim\big(\dots\xr{q_1|}|T_1|\xr{q_0|}|T_0|\big)$.
Since $P_i$ is locally finite, $S_i$ is finite.
Also $|S_i|$, being the cone over $\partial\sigma_i*|\lk(\sigma_i,P_i)|$,
is contractible.
Hence $|T_i|=|S_i|$ is a compact contractible polyhedron.
Thus $T(x)$ is a cell-like compactum.

Let us show that $\xi^{-1}(x)=T(x)$.
Given a simplex or Hilbert simplex $\tau$ of the $\Delta$-space $Y$, let $\tau_i$, 
$\sigma_i$ and $x_\tau$ be given by Lemma \ref{kaul}.
If $x_\tau=x$, then each $x_i\in\sigma_i$ and hence $\sigma_i\in S_i$.
Therefore each $\tau_i\in T_i$ and so $\tau\subset T(x)$.
Conversely, if $\tau\subset T(x)$, then each $\tau_i\in T_i$ and hence each $\sigma_i$ 
meets $|S_i|$.
Therefore $x_\tau=x$.
\end{proof}

\subsection{Issues of dimension}

\begin{proof}[Proof of Theorem \ref{delta}(e)] 
If $\dim X\le n$, then we may assume that each $\dim P_i\le n$ 
(see \cite{M00}*{Remark \ref{book:nerves and cubes}}), and then we get $\dim Y\le n$.
\end{proof}

\begin{proof}[Proof of Theorem \ref{delta}(f)] 
If $X$ is locally finite dimensional, then we may assume that $X$ is a union of 
its closed subsets $X_1\subset X_2\subset\dots$ satisfying $X_i\subset\Int X_{i+1}$ 
such that for each $i$ and $j$ the smallest subcomplex $K_{ij}$ of $P_i$ satisfying 
$p^\infty_i(X_j)\subset|K_{ij}|$ is of dimension $\le j$ 
(see \cite{M00}*{Remark \ref{book:nerves and cubes}}).
Moreover, by the proof of Isbell's theorem 
(see again \cite{M00}*{Remark \ref{book:nerves and cubes}}) we may assume that 
each $P_i$ coincides with its smallest subcomplex $\tilde P_i$ such that 
$p^\infty_i(X)\subset|\tilde P_i|$.

By the minimality each $p_i(|K_{i+1,j}|)\subset |K_{ij}|$, and hence by \ref{property*}
also each $q_i(|L_{i+1,j}|)\subset |L_{ij}|$, where $L_{ij}$ is the subcomplex of $Q_i$ 
subdividing $K_{ij}$.
Let $Y_j=\lim\big(\dots\xr{q_1|}|L_{1j}|\xr{q_0|}|L_{0j}|\big)$.
Then $\dim Y_j\le j$.

Let us show that $Y=\bigcup_{i=1}^\infty Y_j$.
Given a simplex or Hilbert simplex $\tau$ of the $\Delta$-space $Y$, let $\tau_i$, 
$\sigma_i$ and $x_\tau$ be as in Lemma \ref{kaul}.
Then $x_\tau\in X_{j-1}$ for some $j$.
Since $X_{j-1}\subset\Int X_j$, we have $x_\tau\notin X^*_j\bydef X\but\Int X_j$.
Let $K^*_{ij}$ be the smallest subcomplex of $P_i$ such that $p^\infty_i(X^*_j)\subset|K^*_{ij}|$.
Since $X_j\cup X^*_j=X$, we have $K_{ij}\cup K^*_{ij}=\tilde P_i=P_i$ for each $i$.
Since $\{x_\tau\}$ is compact and disjoint from $X^*_j$, $\sigma_i$ must be disjoint 
from $|K^*_{ij}|$ as long as $i$ is large enough.
Hence $\sigma_i\in K_{ij}$ when $i$ is large enough.
Therefore $\tau_i\in L_{ij}$ when $i$ is large enough.
Since each $q_i(\tau_{i+1})=\tau_i$ and each $q_i(|L_{i+1}|)\subset |L_i|$, we actually have 
$\tau_i\in L_{ij}$ for all $i$.
Hence $\tau$ is a simplex of $Y_j$.
Thus $Y=\bigcup_{i=1}^\infty Y_j$.

Let us show that each $Y_{j-1}\subset\Int Y_j$.
Let $L^*_{ij}$ be the subcomplex of $Q_i$ subdividing $K^*_{ij}$.
Let $Y^*_j=\lim\big(\dots\xr{q_1|}|L^*_{1j}|\xr{q_0|}|L^*_{0j}|\big)$.
Since $K_{ij}\cup K^*_{ij}=P_i$, we have $L_{ij}\cup L^*_{ij}=Q_i$ and consequently $Y_j\cup Y^*_j=Y$.
Since $Y^*_j$ is closed, $Y\but Y^*_j$ is open and hence lies in $\Int Y_j$.
So it suffices to show that $Y_{j-1}\subset Y\but Y^*_j$.
Suppose on the contrary that $Y_{j-1}\cap Y^*_j$ contains some point $y$.
Being $\Delta$-subspaces of $Y$, both $Y_{j-1}$ and $Y^*_j$ then contain the minimal simplex $\tau$
of $Y$ containing $y$.
Then each $\tau_i\bydef q^\infty_i(\tau)$ belongs to both $L_{i,j-1}$ and $L^*_{ij}$.
Let $\sigma_i$ and $x_\tau$ be given by Lemma \ref{kaul}.
Then $\sigma_i$ meets both $|K_{i,j-1}|$ and $|K^*_{ij}|$.
Therefore $x_\tau$ belongs to both $X_{j-1}$ and $X^*_{j-1}$.
Since $X_{j-1}\subset\Int X_j$ and $X^*_j=X\but\Int X_j$, we get a contradiction.
\end{proof}

\begin{remark} There is an alternative way of showing that every locally finite dimensional
Polish space $X$ is fine shape equivalent to a locally finite-dimensional Polish $\Delta$-space,
bypassing the proof of Theorem \ref{delta}(f).
By \cite{M00}*{Proposition \ref{book:lc-telescope-a}} $X$ is homotopy equivalent
to a mapping telescope of its finite-dimensional closed subspaces.
Now the proof of Theorem \ref{delta}(e) can be used to show that the latter is fine shape
equivalent to a mapping telescope of finite-dimensional Polish $\Delta$-spaces.
Clearly the latter is a locally finite-dimensional Polish $\Delta$-space.
\end{remark}

\begin{example} When applied to a strongly countably dimensional Polish space $X$,
the construction of Theorem \ref{delta} need not produce a strongly countably dimensional 
$\Delta$-space.
Indeed, let us consider the simplicial complex $\Delta^\infty\bydef \bigcup_{i=1}^\infty\Delta^n$,
where each $\Delta^n$ is the $n$-simplex.
Let $P_0=\Delta^\infty$ and let $P_n=P_{n-1}^*=P_0^{\sktr n}$ 
(see \cite{M00}*{\ref{book:canonical4}}). 
Thus $X=|\Delta^\infty|$ is strongly countably dimensional.
Suppose that the admissible subdivision $P_n'$ of $P_n$ and the simplicial map 
$q_n\:|P_n'|\to |P_n|$ are chosen so that the interior of each $n$-simplex $\sigma$ of $P_n$ 
contains an $n$-simplex $\sigma'$ of $P_n'$, and $q_n$ sends every such $\sigma'$ onto $\sigma$.
Then it is easy to see that every point of $Y$ will be contained in a Hilbert simplex.
\end{example}

\section{$\nabla$-resolution}

Let us call a $\Delta$-space {\it $\Delta$-locally finite} if it can be given by an inverse 
sequence $\dots\to|K_1|\to|K_0|$ of metric simplicial complexes $K_i$ that are locally finite.
Let us note that $\Delta$-resolution of Polish spaces (see Theorem \ref{delta}) always produces 
$\Delta$-locally finite $\Delta$-spaces.

\begin{theorem} \label{nabla}
Let $X$ be a $\Delta$-locally finite $\Delta$-space.

(a) If $\dim X\le n$, then $X$ is fine shape equivalent to a $\nabla$-space $Y$ 
of dimension $\le n$.

Moreover, there is a cell-like perfect map $\zeta\:Y\to X$ which is a fine shape equivalence.

(b) If $X$ is locally finite-dimensional, or more generally decomposable, then it is fine shape 
equivalent to a locally finite-dimensional $\nabla$-space $Y$.

Moreover, in both (a) and (b) $Y$ is $\Delta$-locally finite.
\end{theorem}

In order to prove Theorem \ref{nabla} we first describe and study an auxiliary construction, 
which is purely combinatorial.

\subsection{Non-degenerate resolution of a simplicial complex}

Let us recall that the order complex $\Delta(P)$ of a poset $P$ is the simplicial complex whose 
vertices are the elements of $P$ and whose simplexes are the finite chains of $P$.
Although $\Delta(P)$ is defined as an abstract simplicial complex, it is naturally identified with 
an affine simplicial complex in the vector space $\R[P]$ of all finite linear combinations of 
the elements of $P$.

Let $K$ be a simplicial complex and $K^\flat$ be its barycentric subdivision, which may be regarded 
as the order complex $\Delta\big(FP(K)\big)$ of the poset $FP(K)$ of all nonempty simplexes of $K$ 
ordered by inclusion.
Let $[n]=\{0,\dots,n\}$ and $\N=\{0,1,\dots\}$ be regarded as posets (with the usual linear order),
and let $\Delta^n$ and $\Delta^\infty$ be their order complexes.
Thus $\Delta^n$ is the simplicial complex consisting of the $n$-simplex and all its faces, and 
$\Delta^0\subset\Delta^1\subset\dots$ are subcomplexes of $\Delta^\infty$.
Let $D\:FP(K)\to\N$ be the monotone map defined by assigning to each simplex of $K$ its dimension.
Upon passing to the order complexes we get a simplicial map $d\:K^\flat\to\Delta^\infty$.
Let us note that if $\dim K\le n$, then $d(K^\flat)\subset\Delta^n$.

The order complex $K^\flat\boxtimes\Delta^\infty\bydef \Delta\big(FP(K)\x\N\big)$ is a triangulation of 
the product $|K^\flat|\x|\Delta^\infty|$ such that the projections onto its factors are 
triangulated by simplicial maps $P_\infty\:K^\flat\boxtimes\Delta^\infty\to K^\flat$ and 
$K^\flat\boxtimes\Delta^\infty\to\Delta^\infty$.
For each $n\in\N$ we also have the subcomplex 
$K^\flat\boxtimes\Delta^n\bydef \Delta\big(FP(K)\x[n]\big)$ of $K^\flat\boxtimes\Delta^\infty$.%
\footnote{For the reader who is familiar with cone complexes (see \cite{M00}) we can offer
a more geometric description of these and subsequent constructions.
Let $\Delta^\N$ and $\Delta^{[n]}$ denote the cone complexes corresponding to the posets
$\N$ and $[n]$. 
Then $K^\flat\boxtimes\Delta^\infty$ is the barycentric subdivision $\big(K\x\Delta^\N\big)^\flat$ 
of the cone complex $K\x\Delta^\N$, and similarly 
$K^\flat\boxtimes\Delta^n=\big(K\x\Delta^{[n]}\big)^\flat$.
Let us note that $|K^\flat\boxtimes\Delta^\infty|=|K\x\Delta^\N|=|K|\x|\Delta^\N|$ as spaces
(see \cite{M00}*{Lemma \ref{book:barycentric-admissible} and Theorem \ref{book:cell-product}}).}

Let $\hat K^\infty$ be the subcomplex of $K^\flat\boxtimes\Delta^\infty$ consisting of all its 
simplexes that project non-degenerately into $\Delta^\infty$.
This will be our ``non-degenerate resolution'' of $K$.
Also let $\hat K^n$ be the subcomplex of $K^\flat\boxtimes\Delta^n$ consisting of all its 
simplexes that project non-degenerately into $\Delta^n$.
The joint map $\id\x D\:FP(K)\to FP(K)\x\N$ is an injective monotone map.
Upon passing to the order complexes, it induces a simplicial embedding 
$e\:K^\flat\to K^\flat\boxtimes\Delta^\infty$,
which is a section of the projection $P_\infty\:K^\flat\boxtimes\Delta^\infty\to K^\flat$.%
\footnote{Geometrically, $D$ corresponds to a conical map $\bar D\:K\to\Delta^\N$.
The joint map $\id_K\x\bar D\:K\to K\x\Delta^\N$ is an injective conical map.
Upon passing to the barycentric subdivisions, it induces the simplicial embedding $e$.}
Clearly, $e(K^\flat)\subset\hat K^\infty$.
Moreover, if $\dim K\le n$, then $e(K^\flat)\subset\hat K^n$.

Now here is the point of this construction.
Given a simplicial map $f\:K\to L$, it induces a monotone map 
$FP(f)\x\id_\N\:FP(K)\x\N\to FP(L)\x\N$, and upon passing to the order complexes
we get a simplicial map 
$f^\flat\boxtimes\id_{\Delta^\infty}\:K^\flat\boxtimes\Delta^\infty\to L^\flat\boxtimes\Delta^\infty$.%
\footnote{Geometrically, $f$ induces a conical map 
$f\x\id_{\Delta^\N}\:K\x\Delta^\N\to L\x\Delta^\N$, and upon passing to the barycentric subdivisions
we get the simplicial map $f^\flat\boxtimes\id_{\Delta^\infty}$.}
The latter restricts to a {\it non-degenerate} simplicial map 
$\hat f^\infty\:\hat K^\infty\to\hat L^\infty$.
Moreover, the following diagrams of simplical maps are commutative:
\[\begin{CD}
\hat K^\infty@>\hat f^\infty>>\hat L^\infty\\
@Vp_\infty VV@Vp_\infty VV\\
K^\flat@>f^\flat>>L^\flat
\end{CD}
\qquad\text{ and }\qquad
\begin{CD}
\hat K^\infty@>\hat f^\infty>>\hat L^\infty\\
@Ae AA@Vp_\infty VV\\
K^\flat@>f^\flat>>L^\flat.
\end{CD}\]
When both $K$ and $L$ are of dimension $\le n$, we may replace $\hat K^\infty$, $\hat L^\infty$ and
$\hat f^\infty$ with $\hat K^n$, $\hat L^n$ and $\hat f^n$ in these diagrams.

\begin{example} Suppose that $\dim K=1$.
Then $\hat K^1$ is obtained from $e(K^\flat)$ by attaching a half-edge to every vertex of $e(K^\flat)$.
Given a simplicial map $f\:K\to L$, the induced map of the barycentric subdivisions 
$f^\flat\:K^\flat\to L^\flat$ factors, as noted above, through a non-degenerate simplicial map 
$\hat f^1\:\hat K^1\to\hat L^1$.
Namely, each edge of $K$ that was sent to a vertex of $L$ will now go onto the half-edge attached 
to this vertex.
Also, the half-edges attached to the vertices of $K^\flat$ will go to the half-edges attached
to the corresponding vertices of $L^\flat$.%
\footnote{It may seem that the half-edges attached to the vertices of $K^\flat$ that are 
not vertices of $K$ are superfluous for the purposes of making simplicial maps non-degenerate --- 
and so they are in $\hat K^1$.
But if $K$ is enlarged to a $2$-dimensional complex, these half-edges will no longer be superfluous 
in $\hat K^2$.}
\end{example}

\subsection{Deformation retraction of the non-degenerate resolution}

\begin{theorem} \label{def-retr}
If $\dim K\le n$, then $|\hat K^n|$ strongly deformation retracts onto $|e(K^\flat)|$.

Moreover, for every subcomplex $L$ of $K$, the deformation retraction moves 
$|\hat L^n|$ within itself.
\end{theorem}

\begin{definition}\label{collapse-dr} Suppose that $K$ is a simplicial complex which admits an
elementary simplicial collapse onto its subcomplex $L$.
Thus $K\but L=\{\sigma,\tau\}$, where $\sigma$ is not a face of any simplex of $L$
and $\tau$ is a facet of $\sigma$ which is not a face of any simplex of $L$.

The standard strong deformation retraction of $|K|$ onto $|L|$ is defined as follows.
Let $v$ be the vertex of $\sigma$ not contained in $\tau$ and let $\hat\tau$ be 
the barycenter of $\tau$.
Let $J=v*\hat\tau$ and let $h_t\:J\to J$ be defined by $h_t(x)=(1-t)x+tv$.
Thus $h_t$ is a deformation retraction of $J$ onto $v$.
We have $\sigma=J*\partial\tau$ and $\sigma\cap|L|=v*\partial\tau$.
Thus $h_t*\id_{\partial\tau}$ is a deformation retraction of $\sigma$ onto $\sigma\cap|L|$.
Extending it over $|K|$ by $\id_{|L|}$, we obtain a deformation retraction of $|K|$ onto $|L|$.
\end{definition}

Theorem \ref{def-retr} is follows immediately from the following proposition.

\begin{proposition} \label{collapse}
If $\dim K\le n$, then $\hat K^n$ simplicially collapses onto $e(K^\flat)$.

Moreover, for every subcomplex $L$ of $K$, the collapse moves $\hat L^n$ within itself.
\end{proposition}

\begin{proof} The projection $P_n\:|K^\flat|\x|\Delta^n|\to|K^\flat|$ is
triangulated by a simplicial map $K^\flat\boxtimes\Delta^n\to K^\flat$.
Let $p_n\:|\hat K^n|\to |K^\flat|$ be its restriction to $|\hat K^n|$.
We will first show that for each simplex $\sigma$ of $K^\flat$, the point-inverse 
$p_n^{-1}(\hat\sigma)$ of its barycenter is collapsible.
In order to do that, we in turn need to understand the cell complex that triangulates 
$p_n^{-1}(\hat\sigma)$.

Let $v_0,\dots,v_m$ be the vertices of $\sigma$.
A simplex $\tau$ of $K^\flat\boxtimes\Delta^n$ which projects onto $\sigma$ is of the form
$\tau=\tau_0*\dots*\tau_m$, where each $\tau_i$ projects onto $v_i$.
Clearly, $\tau\cap P_n^{-1}(\hat\sigma)=\tau_0\x\dots\x\tau_m$.
All such cells $\tau_0\x\dots\x\tau_m$ form a cell complex $R=R(\sigma,n)$ cellulating the simplex 
$P_n^{-1}(\hat\sigma)$.
Now $p_n^{-1}(\hat\sigma)$ is cellulated by the subcomplex $Q=Q(\sigma,n)$ of $R$ consisting 
of all its cells $\tau_0\x\dots\x\tau_m$ such that $\tau=\tau_0*\dots*\tau_m$ projects 
non-degenerately into $|\Delta^n|$.

The vertices of $\sigma$ form a chain $v_0<\dots<v_m$ in the poset $FP(K)$.
The vertices of $\tau$ form a chain   
\[\begin{NiceMatrix}
&(v_0,\,j_1)&<&\dots&<&(v_1,\,j_{n_0})\\
<&(v_1,\,j_{n_0+1})&<&\dots&<&(v_1,\,j_{n_0+n_1})\\
\\ \hdottedline\\
<&(v_m,\,j_{n_0+\dots+n_{m-1}+1})&<&\dots&<&(v_m,\,j_{n_0+\dots+n_m}).
\end{NiceMatrix}\]
in the poset $FP(K)\x[n]$.
Thus $0\le j_1\le\dots\le j_{n_0+\dots+n_m}\le n$, and from the strict inequalities between 
the vertices of $\tau$ we get the following more specific inequalities:
\[\begin{NiceMatrix}
{0\quad} &\le &j_1&<&\dots&<&j_{n_0}&\\
&\le&j_{n_0+1}&<&\dots&<&j_{n_0+n_1}&\\
\\ \hdottedline\\
&\le&j_{n_0+\dots+n_{m-1}+1}&<&\dots&<&j_{n_0+\dots+n_m}&\le\quad n.
\end{NiceMatrix}\]
Finally, $\tau$ projects non-degenerately into $|\Delta^n|$ if and only if 
$j_1<\dots<j_{n_0+\dots+n_m}$.
Let us note that $n_0+\dots+n_m$ is the number of vertices of the polytope \
$\tau_1^{n_1-1}\x\dots\x\tau_m^{n_m-1}$, and its dimension is $n_0+\dots+n_m-(m+1)$.

Thus a $k$-cell $C$ of $R$ is determined by an $(m+1)$-tuple $(A_0,\dots,A_m)$ of nonempty 
subsets of $[n]$ of total cardinality $k+m+1$ such that $A_0\le\dots\le A_m$, where $X\le Y$ means 
that $x\le y$ for each $x\in X$ and each $y\in Y$.
A cell $D$ is a face of $C$ if it corresponds to an $(m+1)$-tuple $(B_0,\dots,B_m)$ such that 
each $B_i\subset A_i$.
The cell $C$ belongs to $Q$ if and only if $A_0<\dots<A_m$, where $X<Y$ means that $x<y$ for 
each $x\in X$ and each $y\in Y$.

\begin{example} \label{grayson} Let us see how the subdivision $R$ of the simplex 
$P_n^{-1}(\hat\sigma)$ into products of simplexes looks like in a couple of examples.
(In fact, this subdivision appears to be the same as Grayson's \cite{Gra}*{\S4}; 
see also \cite{EG}*{p.\ 708}.)

For instance, if $n=2$, $m=1$ and $k=2$, we have the following pairs $(A_0, A_1)$ of 
total cardinality $4$: $(\{0\},\,\{0,1,2\})$, $(\{0,1\},\,\{1,2\})$ and $(\{0,1,2\},\,\{2\})$. 
They correspond to the $2$-cells of a cellular subdivision of $\Delta^2$: two homothetic copies 
of $\Delta^2$ at the initial and terminal vertices of $\Delta^2$ and a square in the middle.

If $n=3$, $m=1$ and $k=3$, we have the following pairs $(A_0, A_1)$ of total cardinality $5$:
$(\{0\},\,\{0,1,2,3\})$, $(\{0,1\},\,\{1,2,3\})$, $(\{0,1,2\},\,\{2,3\})$ and 
$(\{0,1,2,3\},\,\{3\})$.
They correspond to the $3$-cells of a cellular subdivision of $\Delta^3$: two homothetic copies 
of $\Delta^3$ at the initial and terminal vertices of $\Delta^3$ and two triangular prisms in 
the middle.
\end{example}

\begin{lemma} \label{collapse2}
$Q$ is collapsible.
\end{lemma}

\begin{proof}
We may assume that $n>m$, for otherwise $Q$ is a point and there is nothing to prove.
Let $C$ be a cell of $Q$ and let $(A_0,\dots,A_m)$ be the corresponding tuple of subsets 
of $[n]$.
Let $\lambda=\lambda_C$ be the maximal number such that $A_i=\{i\}$ for all $i<\lambda$.
If $\lambda=m+1$, let us call $C$ {\it terminal}.
Clearly, there is precisely one terminal cell in $Q$, and it is a $0$-cell.
(In fact, this is vertex is nothing but the point $Q(\sigma,m)$ which we have just 
excluded from consideration.)
Assuming that $\lambda\le m$, let us call $C$ {\it excessive} if $\lambda\in A_\lambda$ and 
{\it deficient} otherwise.
If $C$ is excessive, let $C^-$ be the cell of $Q$ corresponding to $(A_0^-,\dots,A_m^-)$, 
where $A_i^-=A_i$ for $i\ne\lambda$ and $A_\lambda^-=A_\lambda\but\{\lambda\}$.
If $C$ is deficient, let $C^+$ be the cell of $Q$ corresponding to $(A_0^+,\dots,A_m^+)$, 
where $A_i^+=A_i$ for $i\ne\lambda$ and $A_\lambda^+=A_\lambda\cup\{\lambda\}$.
Let us note that since $n>m$, all $n$-cells of $Q$ are excessive and all non-terminal $0$-cells of 
$Q$ are deficient.

Let $Q^k$ be the subcomplex of $Q$ consisting of all cells of the skeleton $Q^{(k-1)}$ and of 
all excessive and terminal $k$-cells of $Q$.
Thus $Q^n=Q$ and $Q^0$ consists only of the terminal cell.
To show that $Q$ is collapsible it suffices to show that $Q^k$ collapses onto $Q^{k-1}$ for each $k>0$.
Assuming that $k>0$, let $Q^k_l$ be the subcomplex of $Q^k$ consisting of all cells of $Q^{k-1}$ 
and of all cells $C$ in $Q^k\but Q^{k-1}$ such that $\lambda_C\ge l$.
Thus $Q^k_0=Q^k$ and $Q^k_{m+1}=Q^{k-1}$.
To show that $Q^k$ collapses onto $Q^{k-1}$ it suffices to show that $Q^k_l$ collapses onto 
$Q^k_{l+1}$ for each $l\le m$.

Let us note that $Q^k_l\but Q^k_{l+1}$ consists of all excessive $k$-cells $C$ of $Q$ 
such that $\lambda_C=l$ and of all deficient $(k-1)$-cells $C$ of $Q$ such that $\lambda_C=l$.
Let $C$ be a cell of the first type.
Then $C^-$ is defined and is a cell of $Q^k_l$ of the second type.
Indeed, let $(A_0,\dots,A_m)$ and $(A_0^-,\dots,A_m^-)$ be the corresponding tuples of subsets of $[n]$.
We have $A_i^-=A_i=\{i\}$ for all $i<l$ and $l\notin A_l^-=A_l\but\{l\}$.
Hence $\lambda_{C^-}=l$ and $C^-$ is deficient.

Suppose that $C^-$ is a face of some other $k$-cell $D$ of $Q^k_l$.
Let $(B_0,\dots,B_m)$ be the corresponding tuple of subsets of $[n]$.
Writing $A^-=A^-_0\cup\dots\cup A^-_m$ and $B=B_0\cup\dots\cup B_m$, we have
$B=A^-\cup\{j\}$ for some $j\notin A^-$.
Since $\{0,\dots,l-1\}\subset A^-\subset B$, we have $j\ge l$.
If $j=l$, then $j\in B_{l-1}$ due to $D\ne C$.
Hence $\lambda_D=l-1$.
Therefore $D\notin Q^k_l$, which is a contradiction.
If $j>l$, then $B_i=A_i^-=\{i\}$ for all $i<l$, and $l\notin B_l$.
Hence $\lambda_D=l$ and $D$ is deficient.
Since $D$ is a deficient $k$-cell, $D\notin Q^k$.
In particular, $D\notin Q^k_l$, which is a contradiction.

Thus $C^-$ is a free face of $C$ in $Q^k_l$.
So $Q^k_l$ collapses onto its subcomplex $\tilde Q^k_l\bydef Q^k_l\but\{C,C^-\}$.
Given another $k$-cell $D$ of $Q^k_l$, by the above $C^-$ is not a face of $D$.
In particular $D^-\ne C^-$, and so $D^-\in\tilde Q^k_l$.
Since $D^-$ is a free face of $D$ in $Q^k_l$, it is also a free face of $D$ in $\tilde Q^k_l$.
Hence $\tilde Q^k_l$ collapses onto $\tilde Q^k_l\but\{D,D^-\}$.
By continuing in this fashion we obtain a collapse of $Q^k_l$ onto $Q^k_{l+1}$.
\end{proof}

We return to the proof of Proposition \ref{collapse}.
Let $\hat K^n_m=p_n^{-1}\big((K^\flat)^{(m)}\big)\cup e(K^\flat)$.
Thus $\hat K^n_n=\hat K^n$ and $\hat K^n_{-1}=e(K^\flat)$.
To obtain the desired collapse of $\hat K^n$ onto $e(K^\flat)$ it suffices to collapse 
each $\hat K^n_m$ onto $\hat K^n_{m-1}$.
For that it in turn suffices to collapse the subcomplex of $\hat K^n$ triangulating
$p_n^{-1}(\sigma)$ onto the subcomplex triangulating $e(\sigma)\cup p_n^{-1}(\partial\sigma)$
for each $m$-simplex $\sigma$ of $K^\flat$.
But such a simplicial collapse is yielded in the obvious way by the collapse of 
$Q=Q(\sigma,n)$ onto its terminal vertex $Q(\sigma,m)$, which is provided by Lemma \ref{collapse2}.
\end{proof}

\subsection{Finite-dimensional $\nabla$-resolution}

\begin{proof}[Proof of Theorem \ref{nabla}(a)]
Let $X$ be given by an inverse sequence $\dots\xr{p_1}|K_1|\xr{p_0}|K_0|$
with locally finite $K_i$.
By the proof of Lemma \ref{surjective}(a) we may assume that the $p_i$ are surjective.
Then each $\dim K_i\le n$.
By Lemma \ref{inverse subdivision} $X$ is homeomorphic via the identity map to 
a $\Delta$-space $X'$ of the form $\lim\big(\dots\xr{p_1}|K'_1|\xr{p_0}|K'_0|\big)$,
where each $K_i'$ is an admissible subdivision of $K_i$ isomorphic to 
the barycentric subdivision.
Let us write $\hat K_i=\widehat{K_i}^n$.
The embedding $e_i\:K_i'\to\hat K_i$ is a section of the projection 
$\pi_i\:\hat K_i\to K_i'$.
Each $p_i\boxtimes\id_{\Delta^n}\:K_{i+1}'\boxtimes\Delta^n\to K_i'\boxtimes\Delta^n$ 
restricts to a non-degenerate simplicial map $\hat p_i\:\hat K_{i+1}\to\hat K_i$,
where each $\dim\hat K_i\le n$.
Each $K_i'\boxtimes\Delta^n$ is locally finite, hence so is $\hat K_i$.

Thus $Y\bydef \lim\big(\dots\xr{\hat p_1}|\hat K_1|\xr{\hat p_0}|\hat K_0|\big)$ is 
a $\Delta$-locally finite $\nabla$-space of dimension $\le n$.
Since the $\hat p_i$ lift the bonding maps $p_i\:K_{i+1}'\to K_i'$ 
with respect to the projections $\pi_i\:\hat K_i\to K_i'$, we obtain 
a continuous map of the extended mapping telescopes 
$\pi_{[0,\infty]}\:|\hat K|_{[0,\infty]}\to |K'|_{[0,\infty]}$,
which restricts to a map $\zeta\:Y\to X'$.
The point-inverses of the complementary map
$\pi_{[0,\infty)}\:|\hat K|_{[0,\infty)}\to |K'|_{[0,\infty)}$
are compact polyhedra, each homeomorphic to the point-inverse of
the barycenter of some simplex of $K_i'$ in $|\hat K_i|$.
Hence by Lemma \ref{collapse2} they are collapsible.
Therefore the point-inverses of $\zeta$, being inverse limits of these,
are compact and cell-like.

In order to prove that $\zeta$ is a fine shape equivalence, it is convenient
to use a different representation of $X'$ as an inverse limit of simplicial maps
between metric simplicial complexes.
Let $Q_i=|\hat K_i|$ and let $q_i\:Q_{i+1}\to Q_i$ denote the composition
$|\hat K_{i+1}|\xr{\pi_{i+1}}|K_{i+1}'|\xr{p_i}|K_i'|\xr{e_i}|\hat K_i|$.
Since the composition $|K_i'|\xr{e_i}|\hat K_i|\xr{\pi_i}|K_i'|$ 
equals the identity, we have $X'=\lim\big(\dots\xr{q_1}Q_1\xr{q_0}Q_0\big)$.
Moreover, if $\rho_i\:|\hat K_i|\to Q_i$ denotes the composition
$|\hat K_i|\xr{\pi_i}|K_i'|\xr{e_i}|\hat K_i|$, then the commutative diagram
\[\begin{CD}
\dots@=Q_i@=|\hat K_{i+1}|@>\hat p_i>>|\hat K_i|@=Q_i@=|\hat K_i|@>\hat p_{i-1}>>\dots\\
@.@V\rho_{i+1}VV@V\pi_{i+1}VV@V\pi_iVV@V\rho_iVV@V\pi_iVV@.\\
\dots@>e_{i+1}>>|\hat K_{i+1}|@>\pi_{i+1}>>|K_{i+1}'|@>p_i>>|K_i'|
@>e_i>>|\hat K_{i+1}|@>\pi_i>>|K_i'|@>p_{i-1}>>\dots
\end{CD}\]
implies that the $\rho_i$ commute with the bonding maps $q_i$ and $p_i$, and hence
induce a continuous map of the extended mapping telescopes 
$\rho_{[0,\infty]}\:|\hat K|_{[0,\infty]}\to Q_{[0,\infty]}$ which in fact restricts
to the same map $\zeta\:Y\to X'$.

Theorem \ref{def-retr} yields a deformation retraction of each $|\hat K_i|$ 
onto $e_i(|K_i'|)$, that is, a homotopy $h_t^i\:|\hat K_i|\to|\hat K_i|$ from 
$h_0^i=\id$ to a retraction $h_1^i$ of $|\hat K_i|$ onto $e_i(|K_i'|)$.
Then the composition 
$\tilde h_t^i\:|\hat K_i|\xr{h_t^i}|\hat K_i|\xr{\pi_i}|K_i'|\xr{e_i}|\hat K_i|$
is a homotopy from $\rho_i$ to $h_1^i$.
By combining $h^i_t$ and $\tilde h^i_{1-t}$ we get a homotopy
$H_t^i\:|\hat K_i|\to|\hat K_i|$ from $\id$ to $\rho_i$.
Then the composition
$\bar H_t^i\:|\hat K_{i+1}|\xr{\hat p_i}|\hat K_i|\xr{H_t^i}|\hat K_i|$
is a homotopy from $\hat p_i$ to the composition
$|\hat K_{i+1}|\xr{\hat p_i}|\hat K_i|\xr{\pi_i}|K_i'|\xr{e_i}|\hat K_i|$,
which in turn equals $q_i$ (using that $\hat p_i$ lifts $p_i$).
Let us define a map $\phi_{[i,\,i+1]}\:MC(q_i)\to MC(\hat p_i)$ by composing 
the homeomorphism $MC(q_i)\to |\hat K_{i+1}|\x I\cup MC(q_i)$ with 
the union of the metric quotient map $|\hat K_{i+1}|\x I\to MC(\hat p_i)$ 
and the map $MC(q_i)\to |\hat K_i|$ to which $\bar H_t^i$ descends.
These combine into a continuous map 
$\phi_{[0,\infty)}\:Q_{[0,\infty)}\to |\hat K|_{[0,\infty)}$.

Since $\rho_i\rho_i=\rho_i$, the compositions 
$|\hat K_i|\xr{h_t^i}|\hat K_i|\xr{\rho_i}|\hat K_i|$ 
and $|\hat K_i|\xr{\tilde h_t^i}|\hat K_i|\xr{\rho_i}|\hat K_i|$ are equal.
Hence the self-homotopy $|\hat K_i|\xr{H_t^i}|\hat K_i|\xr{\rho_i}|\hat K_i|$
of $\rho_i$ is null-$2$-homotopic.
It follows that the composition 
$Q_{[0,\infty)}\xr{\phi_{[0,\infty)}}|\hat K|_{[0,\infty)}
\xr{\rho_{[0,\infty)}}Q_{[0,\infty)}$
is homotopic to the identity map by a homotopy $\Phi_t$.
Also, since $h_t^i$ keeps $e_i(|K_i'|)$ fixed, the compositions 
$|\hat K_i|\xr{\rho_i}|\hat K_i|\xr{h_t^i}|\hat K_i|$ and 
$|\hat K_i|\xr{\rho_i}|\hat K_i|\xr{\tilde h_t^i}|\hat K_i|$ are both equal to $\rho_i$
for each $t$.
Hence the self-homotopy $|\hat K_i|\xr{\rho_i}|\hat K_i|\xr{H_t^i}|\hat K_i|$ of $\rho_i$ 
stays at $\rho_i$ for each $t$.
It follows that the composition 
$|\hat K|_{[0,\infty)}\xr{\rho_{[0,\infty)}}Q_{[0,\infty)}\xr{\phi_{[0,\infty)}}|\hat K|_{[0,\infty)}$
is homotopic to the identity map by a homotopy $\Psi_t$.

Let $x_k$ be a sequence of points in $Q_{[0,\infty)}$ converging to an $x_\infty\in X'$.
Since the $K_i$ and hence also the $\hat K_i$ are locally finite, the proof of 
Lemma \ref{convergence} works to show that $M\bydef \{x_k\mid k\in\N\}$ lies in 
the mapping telescope $|C|_{[0,\infty)}$ of an inverse sequence of the form 
$\dots\xr{q_1|}|C_1|\xr{q_0|}|C_0|$, where $C_i\subset\hat K_i$ are finite subcomplexes 
such that each $q_i(|C_{i+1}|)\subset |C_i|$.
Since each $q_i$ is a lift of $p_i$, we have $p_i(|D_{i+1}|)\subset |D_i|$, where 
$D_i=\pi_i(C_i)$.
Consequently $q_i(|\hat D_{i+1}|)\subset |\hat D_i|$, where $\hat D_i=\widehat{D_i}^n$.
Since $\hat p_i$ is also a lift of $p_i$, we further have
$\hat p_i(|\hat D_{i+1}|)\subset |\hat D_i|$.
Since the deformation retraction $h_t^i$ moves $|\hat D_i|$ only within itself, 
$\phi_{[0,\infty)}$ sends the mapping telescope of
$\dots\xr{q_1|}|\hat D_1|\xr{q_0|}|\hat D_0|$ into the mapping telescope
$|\hat D|_{[0,\infty)}$ of $\dots\xr{\hat p_1|}|\hat D_1|\xr{\hat p_0|}|\hat D_0|$.
Since each $D_i$ is finite, $|\hat D|_{[0,\infty]}$ is compact.
Thus $\phi_{[0,\infty)}$ is an $X'$-$Y$-approaching map.

Similar arguments show that $\Phi_t$ is an $X'$-$X'$-approaching homotopy and
$\Psi_t$ is a $Y$-$Y$-approaching homotopy.
\end{proof}

\subsection{General $\nabla$-resolution}

\begin{lemma} \label{telescope-collapse}
Let $K_0\subset K_1\subset\dots$ be subcomplexes of $K$ such that $\dim K_i\le i$.
Then the mapping telescope $|\hat K|^{[0,\infty)}$ of the inclusions 
$|\widehat{K_0}^0|\subset|\widehat{K_1}^1|\subset\dots$
strongly deformation retracts onto the mapping telescope $|e(K^\flat)|_{[0,\infty)}$ of the inclusions
$|e(K^\flat_0)|\subset |e(K^\flat_1)|\subset\dots$.

Moreover, the deformation retraction moves $|\hat L^i|\x[i,i+1]\cup|\hat L^{i+1}|\x\{i+1\}$ 
within itself for every subcomplex $L$ of $K_i$.
\end{lemma}

\begin{proof} By Proposition \ref{collapse} each $|\widehat{K_n}^n|$ 
strongly deformation retracts onto $|e(K^\flat_n)|$.
Hence $|\widehat{K_n}^n|\x [n,\,n+1]$ strongly deformation retracts onto 
$|e(K^\flat_n)|\x [n,\,n+1]\cup|\widehat{K_n}^n|\x\{n,n+1\}$.
Thus $|\hat K|^{[0,\infty)}$ strongly deformation retracts onto 
$|e(K^\flat)|_{[0,\infty)}\cup\bigcup_n|\widehat{K_n}^n|\x\{n\}$.
The latter in turn strongly deformation retracts onto $|e(K^\flat)|_{[0,\infty)}$.
\end{proof}

\begin{proof}[Proof of Theorem \ref{nabla}(b)]
Let $X$ be given by an inverse sequence $\dots\xr{p_1}|K_1|\xr{p_0}|K_0|$
with locally finite $K_i$.
By the proof of Lemma \ref{surjective}(a) we may assume that the $p_i$ are surjective.
Since $X$ is decomposable, by Theorem \ref{telescope-decomposition}(a) 
$X$ is homotopy equivalent to the mapping telescope $X_{[0,\infty)}$ of 
its $\Delta$-subspaces $X_0\subset X_1\subset\dots$ such that each $\dim X_n\le n$.

Let $L_{in}$ be the smallest subcomplex of $K_i$ such that $|L_{in}|$ contains 
$p^\infty_i(X_n)$.
By the proof of Lemma \ref{delta-subspaces}(b) $X_n=\lim|L_{in}|$ and each
$|L_{in}|=p^\infty_i(X_n)$.
In particular, each $p_i(|L_{i+1,\,n}|)=|L_{in}|$ and hence $\dim L_{in}\le n$.
Let $|L|_{i,\,[0,\infty)}$ be the mapping telescope of the inclusions
$|L_{i0}|\subset |L_{i1}|\subset\dots$.
Then $X_{[0,\infty)}=\lim\big(\dots\to|L|_{1,\,[0,\infty)}\to|L|_{0,\,[0,\infty)}\big)$.

Let $|\hat L|_i^{[0,\infty)}$ be the mapping telescope of the inclusions
$|\widehat{L_{i0}}^0|\subset |\widehat{L_{i1}}^1|\subset\dots$ and let
$Y_{[0,\infty)}=\lim\big(\dots\to|L|_1^{[0,\infty)}|\to|L|_0^{[0,\infty)}\big)$.
The projections $|\hat L|_i^{[0,\infty)}\to|L|_{i,\,[0,\infty)}$
commute with the bonding maps, and so yield a continuous map
$|\hat L|_{[0,\infty]}^{[0,\infty)}\to|L|_{[0,\infty],\,[0,\infty)}$
of the extended mapping telescopes, which restricts to a map 
$\zeta\:Y_{[0,\infty)}\to X_{[0,\infty)}$.

By Lemma \ref{inverse subdivision} $X$ is homeomorphic via the identity map to 
a $\Delta$-space $X'$ of the form $\lim\big(\dots\xr{p_1}|K'_1|\xr{p_0}|K'_0|\big)$,
where each $K_i'$ is an admissible subdivision of $K_i$ isomorphic to 
the barycentric subdivision $K_i^\flat$.
In fact, by the proof of Lemma \ref{inverse subdivision} $K_i'=(K_i^+)^\flat$, 
where $K_i^+$ is an apexed simplicial complex obtained from $K_i$ by endowing its 
simplexes with apexes located at certain points.
Let $L'_{in}$ be the subcomplex of $K_i'$ subdividing $L_{in}$.

Lemma \ref{telescope-collapse} yields a deformation retraction of each 
$|\hat L|_i^{[0,\infty)}$ onto the mapping telescope $|e(L')|_{i,\,[0,\infty)}$ 
of the inclusions $|e(L_{i0}')|\subset |e(L_{i1}')|\subset\dots$.
Using these deformation retractions, it follows similarly to the proof of (a)
that $\zeta$ is a fine shape equivalence.

Let $R$ be the triangulation of $[0,\infty)$ with vertices at the integer points.
Then $Q_i\bydef (K_i^+\x R)^\flat$ is a triangulation of the product $|K_i|\x|R|$ 
such that the projections onto its factors are triangulated by simplicial maps
$Q_i\to K_i'$ and $Q_i\to R^\flat$.
Let us note that each $|L_{i,\,[0,\infty)}|$ is triangulated by a subcomplex $T_i$ 
of $Q_i$, and $X'_{[0,\infty)}\bydef \lim\big(\dots\to|T_1|\to|T_0|\big)$ is 
a $\Delta$-space which is homeomorphic to $X_{[0,\infty)}$ via the identity map.

Now $Q_i\boxtimes\Delta^\infty$ is a simplicial complex such that
$|Q_i\boxtimes\Delta^\infty|=|Q_i|\x|\Delta^\infty|$ as spaces 
\cite{M00}*{Theorem \ref{book:cell-product}}, and
$|\hat L|_i^{[0,\infty)}$ is triangulated by its subcomplex $\hat T_i$.
Then $Y'_{[0,\infty)}\bydef \lim\big(\dots\to|\hat T_1|\to|\hat T_0|\big)$ is 
a $\Delta$-space which is homeomorphic to $Y_{[0,\infty)}$ via the identity map.
Since the bonding maps $\hat p_i$ are non-degenerate, $Y'_{[0,\infty)}$ is 
a $\nabla$-space.
Clearly it is locally finite-dimensional and $\Delta$-locally finite.
\end{proof}

\begin{remark} The space $Y_{[0,\infty)}$ produced by the proof of 
Theorem \ref{nabla}(b) is in fact the mapping telescope of the inclusions
$Y_0\subset Y_1\subset\dots$ between the $\Delta$-subspaces 
$Y_n\bydef \lim\big(\dots\to|\widehat{L_{1n}}^n|\to|\widehat{L_{0n}}^n|\big)$
of the $\nabla$-space 
$Y\bydef \lim\big(\dots\to|\widehat{K_1}^\infty|\to|\widehat{K_0}^\infty|\big)$.
Since the $\Delta$-subspaces $X_i$ of $X$ are actually normal, it is easy to see
that the $Y_i$ are also normal $\Delta$-subspaces of $Y$.
Also it is easy to see that $Y=\bigcup_i Y_i$.
Hence $Y$ is homotopy equivalent to $Y_{[0,\infty)}$ by 
Theorem \ref{telescope-decomposition}(a).

Thus $Y$ is another $\nabla$-space which is fine shape equivalent to $X$.
Although the construction of $Y$ is more straightforward than that of 
$Y'_{[0,\infty)}$, it should be noted that $Y$ is not locally finite-dimensional,
nor even $\Delta$-locally finite, in contrast to $Y'_{[0,\infty)}$.
\end{remark}

\subsection{The case of unbounded dimension}

This subsection is not used the rest of the paper; it is devoted to the proof of the following

\begin{theorem}\label{def-retr2}
If $K$ is finite-dimensional, then $|\hat K^\infty|$ strongly deformation retracts onto $|e(K^\flat)|$.

Moreover, for each $n$ and every subcomplex $L$ of $K$ such that $\dim L\le n$, the deformation retraction moves 
$|\hat L^n|$ only within itself.
\end{theorem}

This may well be true when $K$ is infinite-dimensional, but if so, that must be more difficult to prove 
(see Example \ref{infinite collapse}).

\begin{proposition}\label{collapse'}
If $\dim K\le n$ and $M$ is subcomplex of $K$ with $\dim M\le n-1$, then $\hat K^n$ 
simplicially collapses onto $\hat M^{n-1}\cup e(K^\flat)$. 

Moreover, for every subcomplex $L$ of $K$ the collapse moves $\hat L^n$ within itself.
\end{proposition}

\begin{proof}
Following the same approach as in the proof of Proposition \ref{collapse}, it suffices to show 
that each $Q(n)=Q(\sigma,n)$ collapses onto $Q(n-1)$, where $n>m=\dim\sigma$.
Following the proof of Lemma \ref{collapse2}, the latter in turn reduces to showing that each 
$Q^k_l(n)\cup Q(n-1)$ collapses onto $Q^k_{l+1}(n)\cup Q(n-1)$.
This is proved just like in the proof of Lemma \ref{collapse2}, using that
a cell $C\in Q(n)\but Q(n-1)$ is never a face of a cell in $Q(n-1)$,
and, if excessive, also satisfies $C^-\notin Q(n-1)$ due to $\lambda_C\le m<n$.
Indeed, the inequality $\lambda_C\le m$ fails only when $C$ is the terminal $0$-cell,
but the terminal $0$-cell belongs to $Q(m)\subset Q(n-1)$ due to $m<n$.
\end{proof}

\begin{definition} \label{infinite collapse-def}
Suppose that a space $X_\infty$ is the union of a chain $X_0\subset X_1\subset\dots$
of its subspaces, where each $X_{i+1}$ admits a strong deformation retraction onto $X_i$.
We will construct the {\it infinite composition} of these deformation retractions, which 
will be a strong deformation retraction of $X_\infty$ onto $X_0$ as long as it is known 
to be continuous.

Thus for each $i\ge 0$ we are given a homotopy $h_t^i\:X_{i+1}\to X_{i+1}$, $t\in [0,1]$, 
starting from $h_0^i=\id$ and terminating with $h_1^i\:X_{i+1}\xr{r_i}X_i\subset X_{i+1}$, 
where $r_i$ is a retraction.

Let $r^k_i$ denote the composition $X_k\xr{r_{k-1}}\dots\xr{r_i}X_i$ whenever $k\ge i\ge 0$.
Let us note that $r^k_i|_{X_j}=r^j_i$ whenever $k\ge j\ge i\ge 0$.
Hence $r^\infty_i\:X_\infty\to X_i$, defined by $r^\infty_i(x)=r^k_i(x)$ whenever 
$x\in X_k$, is well-defined (as a set-theoretic map).
Clearly, the composition $X_\infty\xr{r^\infty_{i+1}}X_{i+1}\xr{r^{i+1}_i}X_i$ equals 
$r^\infty_i$.

Let $H_t\:X_\infty\to X_\infty$, $t\in [0,1]$, be defined as follows.
We set $H_0=\id$ and $H_{2^{-i}}=r^\infty_i$ for each $i\ge 0$.
Let $\phi_i\:[0,1]\to[2^{-i-1},2^{-i}]$ be the increasing linear homeomorphism,
and let us define $H_{\phi_i(t)}$ (as a set-theoretic map) to be the composition 
$X_\infty\xr{r^\infty_{i+1}}X_{i+1}\xr{h_t^i}X_{i+1}$.
Let $H\:X_\infty\x I\to X_\infty$ be defined by $H(x,t)=H_t(x)$.
If we are somehow able to show that $H$ is continuous, then clearly $H$ will be 
the desired strong deformation retraction of $X_\infty$ onto $X_0$.
\end{definition}

\begin{lemma} \label{continuity}
Suppose that for each $i=0,1,\dots$ the restriction 
$H|_{X_\infty\x[0,2^{-i}]}$ is continuous at each point of $X_i\x[0,2^{-i}]$.
Then $H$ is continuous.
\end{lemma}

Let us note that $H_t$ for $t\in [0,2^{-i}]$ keeps $X_i$ fixed.

\begin{proof} Let us show that each $r^\infty_i$ is continuous.
Given an $x\in X_\infty$, it must belong to some $X_k$.
By the hypothesis $r^\infty_i=H_{2^{-i}}$ is continuous at each point of $X_i$.
If $k\le i$, we get that $r^\infty_i$ is continuous at $x$, as desired.
If $k\ge i$, we get that $r^\infty_k$ is continuous at $x$.
Also $r^k_i$ is continuous, so their composition $r^\infty_i$ must be continuous at $x$.

Thus each $r^\infty_i$ is continuous.
Since $h_t^i$ is also continuous, we get that each $H|_{X_\infty\x[2^{-i-1},2^{-i}]}$ 
is continuous.
Since $r^{i+1}_ir^\infty_{i+1}=r^\infty_i$, we obtain that $H|_{X_\infty\x(0,1]}$ 
is continuous.

Given an $x\in X_\infty$, let us show that $H$ is continuous at $(x,0)$.
We have $x\in X_i$ for some $i$.
Then by the hypothesis $H|_{X_\infty\x[0,2^{-i}]}$ is continuous at $(x,0)$.
Hence so is $H$.
\end{proof}

Theorem \ref{def-retr2} can be deduced from its version where all polyhedra are retopologized 
with weak topology (see \cite{M00}*{Theorem \ref{book:weak} and Lemma \ref{book:cof2}} or 
alternatively \cite{M00}*{proof of Theorem \ref{book:cellular}}) and this version can in turn be 
deduced from Proposition \ref{collapse'} (with $M=K$ and $n>\dim K$) by using a version 
of Lemma \ref{continuity}.

Nevertheless, below we give a direct proof of Theorem \ref{def-retr2} which does not appeal 
to weak topology.
It may still hold some value, particularly in connection with the problem of extending 
Theorem \ref{def-retr2} to the case of an infinite-dimensional $K$.

\begin{example} \label{infinite collapse}
Let $K_\infty$ be the triangulation of $[0,\infty)$ with vertices at the integer points.
Let $K_n$ be its subcomplex triangulating $[0,n]$.
Then each $K_{n+1}$ simplicially collapses onto $K_n$, leading to a strong deformation 
retraction of $[0,n+1]$ onto $[0,n]$.
Their infinite composition is easily seen to be continuous in this case, and thus is
a strong deformation retraction of $|K_\infty|=[0,\infty)$ onto $|K_0|=\{0\}$.

Now let $v*K_\infty$ be the cone over $K_\infty$, triangulating the metric cone over 
$[0,\infty)$ (see \cite{M00}*{\S\ref{book:join}}).
Let $v*K_n$ be its subcomplex triangulating the metric cone over $[0,n]$.
Then each $v*K_{n+1}\cup K_\infty$ simplicially collapses onto $v*K_n\cup K_\infty$, 
leading to a strong deformation retraction of $v*[0,n+1]\cup[0,\infty)$ onto 
$v*[0,n]\cup [0,\infty)$.

It turns out that their infinite composition $H_t$ is not continuous.
Indeed, let $r^j_i$ be the retraction of $|CK_j\cup K_\infty|$ onto $|CK_i\cup K_\infty|$
resulting from the finite composition of collapses, and let $r^\infty_i$
be defineed as above.
Let $v_n$ be the vertex of $K_\infty$ at the point $n\in[0,\infty)$, and
let $x_i=(1-2^{-i})v+2^{-i}v_i\in v*v_i$.
Then it is easy to see that $r^{i+1}_i(x_{i+1})=x_i$.
Hence $r^i_0(x_i)=x_0=v_0$. 
Consequently also $r^\infty_0(x_i)=v_0$.
However $x_i\to v$ as $i\to\infty$.
Thus $H_1=r^\infty_0$ is discontinuous.
\end{example}

In view of Example \ref{infinite collapse} it seems to be challenging to deduce Theorem \ref{def-retr2} 
directly from Proposition \ref{collapse'} by using the infinite composition.
We will take a different route, by rearranging the collapses in this complicated infinite composition 
so as to obtain a finite number of infinite compositions that are simple enough to be continuous.

\begin{proof}[Proof of Theorem \ref{def-retr2}]
We will use the notation in the proof of Propositions \ref{collapse} and \ref{collapse'} 
and Lemma \ref{collapse2}.
In particular, we fix a simplex $\sigma$ of $K^\flat$ and set $m=\dim\sigma$.
Let us consider the infinite cell complex $S=S(\sigma)=\bigcup_{n\ge m}Q(\sigma,n)$.
Let $S_l$ denote the set of all cells $C$ of $S$ such that $\lambda_C\ge l$.
It is easily seen to be a subcomplex of $S$.
We have $S_0=S$ and $S_{m+1}=Q(m,m)$, which consists only of the terminal cell.

Let $S_l(n)$ consist of $S_{l+1}$ and of all cells $C$ of $S_l$ such that $C\in Q(n)$.
It is easily seen to be a subcomplex of $S_l$.
Also let $S_l(\infty)=\bigcup_{n\ge m} S_l(n)$.
We have $S_l(\infty)=S_l$ and $S_l(m)=S_{l+1}$.

Let $S_l^k(n)$ consist of $S_l(n-1)$ and of all cells $C$ of $S_l(n)$ 
such that either $C\in Q^{(k-1)}$ or $C$ is an excessive or terminal $k$-cell of $Q$.
It is easily seen to be a subcomplex of $S_l(n)$.
Beware that $S_l^k(n)$ is quite different from $Q_l^k(n)$, although the difference
$S_l^k(n)\but S^{k-1}_l(n)$ is the same as $Q_l^k(n)\but Q^k_{l+1}(n)$.

By the proof of Proposition \ref{collapse'} each $Q^k_l(n)$ collapses onto 
$Q^k_{l+1}(n)$.
But in fact, the same construction also describes a collapse of $S^k_l(n)$ onto 
$S^{k-1}_l(n)$.
This follows just like in the proof Proposition \ref{collapse'}, using that 
a cell $C\in S_l\but S_{l+1}$ is never a face of a cell in $S_{l+1}$, 
and, if excessive, also satisfies $C^-\notin S_{l+1}$.

By combining the collapses of $S^k_l(n)$ onto $S^{k-1}_l(n)$ for $k=n,\dots,m+1$ 
we obtain a collapse of $S^n_l(n)=S_l(n)$ onto $S^m_l(n)=S_l(n-1)$.
Thus $S_l(n)$ collapses onto $S_l(n-1)$ for each finite $n>m$.

Let $W=W(\sigma)$ be the subcomplex of $\hat K^\infty$ triangulating 
$p_\infty^{-1}(\sigma)$.
Let $\partial W$ be the subcomplex of $W$ triangulating 
$e(\sigma)\cup p_\infty^{-1}(\partial\sigma)$.
Let $W_l$ be the subcomplex of $W$ consisting of $\partial W$ and all simplexes of 
$W\but\partial W$ that meet $|S|=|S(\sigma)|$ in cells of $S_l$.
Thus $W_0=W$ and $W_{m+1}=\partial W$.

Let $W_l(n)$ be the subcomplex of $W_l$ consisting of $\partial W$ and 
all simplexes of $W\but\partial W$ that meet $|S_l|$ in cells of $S_l(n)$.
Thus $W_l(\infty)=W_l$ and $W_l(m)=W_{l+1}$.
Also, $W_l(n)$ collapses onto $W_l(n-1)$ for each finite $n>m$.
Let $H_l\:|W_l|\x I\to|W_l|$ be the infinite composition of the corresponding 
strong deformation retractions.
Let us show that $H_l$ is continuous.

Let $\Phi_l$ be the set of the free faces of all the collapses that are used
in the construction of $H_l$.
Thus $\Phi_l$ consists of all simplexes of $W\but\partial W$ that correspond
to deficient cells $C$ of $Q$ with $\lambda_C=l$.
It is easy to see that $\Phi_l\cup\partial W$ is a subcomplex of $W$.
Let $W_l'$ be the simplicial subdivision of $W_l$ obtained by ``starring at the
barycenters of all simplexes of $\Phi$, in an order of decreasing dimension''.
More precisely, we define $(W_l')^{(k)}$ to be the simplicial subdivision of 
the $k$-skeleton $W_l^{(k)}$ obtained by starring at the barycenters of all 
simplexes of $\Phi^{(k)}$, in an order of decreasing dimension.
This does make sense, and it is easy to see that each $(W_l')^{(k)}$ is
isomorphic to the $k$-skeleton of $(W_l')^{(\kappa)}$ whenever $k<\kappa$.
Then $W_l'$ is defined as the union $\bigcup_{k=0}^\infty(W_l')^{(k)}$.

Now each simplex of $W_l'$ is $H_l$-invariant, i.e.\ goes into itself under
each time instant of the homotopy $H_l$.
Given any point $x\in |W_l'|$, the open star $\ost(\sigma_x,W_l')$ of
the smallest simplex $\sigma_x$ of $W_l'$ containing $x$ is also $H_l$-invariant.
Moreover, if $x\in |W_l'(n)|$, then $x$ is fixed under 
$H_l(n)\bydef H_l|_{|W_l|\x[0,2^{-n}]}$, and it follows that the image of 
$\ost(\sigma_x,W_l')$ under the homothety $h_x^r$ with any ratio $r\in (0,1]$ 
centered at $x$ is $H_l(n)$-invariant.
But such images $h_x^r\big(\ost(\sigma_x,W_l')\big)$, $r=1,\frac12,\frac13,\dots$,
form a base of neighborhoods of $x$ in the metric topology on $|W_l'|$
(see \cite{M00}*{Theorem \ref{book:open star}}).
Moreover, since the barycentric subdivision $W_l^\flat$ is admissible
(see \cite{M00}*{Lemma \ref{book:barycentric-admissible}}) and the vertices
of $W_l'$ form a subset of the vertices of $W_l^\flat$, they have no
cluster points in $|W_l|$, and consequently by the Mine--Sakai theorem 
(see \cite{M00}*{Theorem \ref{book:subdivision}}) $W_l'$ is also admissible, i.e.\
$|W_l'|=|W_l|$ as spaces.
Thus every $x\in |W_l(n)|$ has arbitrarily small $H_l(n)$-invariant 
neighborhoods in the metric topology of $|W_l|$.
Therefore by Lemma \ref{continuity} $H_l$ is continuous.

Thus we have proved that $H_l$ is a deformation retraction of $|W_l|$ onto
$|W_{l+1}|$.
By composing these deformation retractions for $l=0,\dots,m$ we get a deformation
retraction of $|W_0|=|W|$ onto $|W_{m+1}|=|\partial W|$.
Now similarly to the proof of Theorem \ref{def-retr} from these we get 
the desired deformation retraction of $|\hat K^\infty|$ onto $|e(K^\flat)|$.
\end{proof}

\subsection*{Acknowledgements}

I'm grateful to A. Gorelov for stimulating discussions and useful remarks.

\subsection*{Disclaimer}

I oppose all wars, including those wars that are initiated by governments at the time when 
they directly or indirectly support my research. The latter type of wars include all wars 
waged by the Russian state in the last 25 years (in Chechnya, Georgia, Syria and Ukraine) 
as well as the USA-led invasions of Afghanistan and Iraq.

\end{document}